\documentclass[onefignum,onetabnum]{siamart250211}

\usepackage{lipsum}
\usepackage{amsfonts}
\usepackage{graphicx}
\usepackage{epstopdf}
\usepackage{algorithmic}
\ifpdf
  \DeclareGraphicsExtensions{.eps,.pdf,.png,.jpg}
\else
  \DeclareGraphicsExtensions{.eps}
\fi

\newsiamremark{remark}{Remark}
\newsiamremark{hypothesis}{Hypothesis}
\crefname{hypothesis}{Hypothesis}{Hypotheses}
\newsiamthm{claim}{Claim}
\newsiamremark{fact}{Fact}
\crefname{fact}{Fact}{Facts}

\headers{Geometric Generalization of Neural Operators}{}

\title{Geometric Generalization of Neural Operators from a Kernel Integral Perspective\thanks{Submitted to the editors DATE. The authors are in alphabetical order.
\funding{We acknowledge funding support from National Key R\&D Program of China 2025YFA1018700, 
National Natural Science Foundation of China (No.62595771, 12471403, and 12288101), Beijing Natural Science Foundation (No. QY25089), and the Fundamental Research Funds for the Central Universities of China. }}}

\author{
Mingyu Han\thanks{School of Mathematical Sciences, Peking University, Beijing, China (\email{hanmy2023@stu.pku.edu.cn},\email{wyhan@stu.pku.edu.cn},\email{yanshu@stu.pku.edu.cn},\email{jiayi22@stu.pku.edu.cn}).}
\and
Daniel Zhengyu Huang\thanks{Corresponding author. Beijing International Center for Mathematical Research,  Center for Machine Learning Research, Peking University, Beijing, China (\email{huangdz@bicmr.pku.edu.cn}).}
\and
Yuhan Wang\footnotemark[2] 
\and
Yanshu Zhang\footnotemark[2]
\and
Jiayi Zhou\footnotemark[2]
}

\usepackage{amsopn}

\ifpdf
\hypersetup{
  pdftitle={Neural Operator},
  pdfauthor={}
}
\fi

\usepackage{appendix}
\usepackage{amssymb}
\usepackage{subfig}
\usepackage{graphicx}
\usepackage{color}
\usepackage{bm}
\usepackage{enumitem}
\usepackage{bbm}
\usepackage{diagbox} 
\usepackage{makecell} 
\usepackage[normalem]{ulem} 
\usepackage{booktabs}
\usepackage{multirow}
\usepackage{enumitem}

\newcommand{\calO}{\mathcal{O}}
\newcommand{\calG}{\mathcal{G}}
\newcommand{\calD}{\mathcal{D}}
\newcommand{\calK}{\mathcal{K}}
\newcommand{\calL}{\mathcal{L}}

\newcommand{\calP}{\mathcal{P}}
\newcommand{\calQ}{\mathcal{Q}}
\newcommand{\calJ}{\mathcal{J}}
\newcommand{\bK}{\mathbb{K}}
\newcommand{\bR}{\mathbb{R}}
\newcommand{\bZ}{\mathbb{Z}}
\newcommand{\dd}{\mathrm{d}}
\newcommand{\bE}{\mathbb{E}}

\begin{document}

\maketitle

\begin{abstract}
Neural operators are neural network-based surrogate models for approximating partial differential equation solution operators, enabling efficient many-query computations in science and engineering when low-to-moderate accuracy is sufficient. Many applications, including engineering design, involve variable and often nonparametric geometries, for which generalization to unseen shapes remains a central practical challenge.
In this work, we adopt a kernel-integral perspective motivated by classical boundary integral formulations and study operator learning on variable geometries through the approximation of geometry-dependent kernel operators, including singular kernels. This perspective clarifies a mechanism for geometric generalization for fixed linear operators and reveals a direct connection between operator learning and fast kernel summation methods. Leveraging this connection, we propose a multiscale point cloud neural operator inspired by Ewald summation, combining Fourier long-range interactions with local geometry-aware corrections. 
We further establish approximation guarantees for the resulting multiscale representation of linear operators defined by singular kernels.
Numerical experiments demonstrate robust generalization across diverse geometries for several commonly used kernels, Laplace-type boundary integral maps, and a large-scale three-dimensional nonlinear vehicle flow example.
\end{abstract}



\section{Introduction}
This paper develops neural network–based surrogate models for partial differential equations (PDEs) defined on variable geometries. Such surrogates are useful in settings that require rapid repeated PDE evaluations across changing domains. Representative applications include engineering design~\cite{amsallem2015design,li2022machine,luo2025efficient,shen2025vortexnet}, where fast screening of candidate geometries is essential, and biomedical applications~\cite{gao2021phygeonet,yin2024dimon,zhou2024ai,guo2025warm}, where a single model is expected to generalize across patient-specific anatomies (e.g., blood flow and tissue deformation). In such many-query settings, the offline costs of data generation and model training can be amortized over a large number of comparatively inexpensive online evaluations.

We consider PDE-induced operators of the form
\begin{align}
\label{eq:G-continuous-intro}
    \mathcal{G}^{\dagger} : (f, \calD) \mapsto u,
\end{align}
where $\calD$ denotes the geometry on which the input and output fields are
represented, $f:\calD\rightarrow\bR^{d_f}$ denotes the prescribed input
data, and $u:\calD\rightarrow\bR^{d_u}$ denotes the corresponding output
field. 
Our primary focus is the boundary setting $\calD=\partial\Omega$. The underlying PDE may be posed in
$\Omega$ or in the exterior domain
$\mathbb{R}^d\setminus\overline{\Omega}$, but the learned operator acts only
on quantities defined on $\partial\Omega$.
This setting arises, for example, when predicting aerodynamic quantities on the
surface of a vehicle.

The goal of a neural operator is to approximate $\mathcal{G}^{\dagger}$ from data $\{f_i,\calD_i, u_i = \calG^{\dagger}(f_i,\calD_i)\}_{i=1}^{n}$.
Whereas much of the existing neural operator literature focuses on variations in the input field over a fixed domain
\cite{zhu2018bayesian,khoo2019switchnet,li2020fourier,lu2021learning},
this work focuses on variations in the geometry $\calD$. In
particular, $\calD$ need not admit a finite-dimensional parameterization, may vary substantially across samples, and may even undergo changes in topology.

To motivate our approach, consider potential flow around a three-dimensional object $\Omega \subset \bR^{3}$. For an inviscid,
incompressible, and irrotational flow, the velocity field can be written as
$v=\nabla\Phi$, where the velocity potential $\Phi$ satisfies the exterior
Laplace problem
\begin{equation}
\label{eq:laplace_exterior}
\begin{aligned}
    \Delta\Phi
    &=
    0
    &&\text{in }\bR^3\setminus\overline{\Omega},\\
    \frac{\partial\Phi}{\partial n_x}
    &=
    0
    &&\text{on }\partial\Omega,\\
    \nabla\Phi(x)
    &\rightarrow
    v_\infty
    &&\text{as }\lVert x\rVert_2\rightarrow\infty,
\end{aligned}
\end{equation}
where $n_x$ denotes the outward unit normal to $\partial\Omega$ and
$v_\infty$ is the prescribed far-field velocity. The quantity of interest is
the surface pressure coefficient, so the corresponding aerodynamic map is
\begin{equation}
    \label{eq:aerodynamic-map}
    \calG^{\dagger}: (v_{\infty}, \partial \Omega) \rightarrow \left.C_p\right|_{\partial\Omega}.
\end{equation}
In the notation of \cref{eq:G-continuous-intro}, the far-field velocity may
be represented as the constant boundary input $f(x)\equiv v_\infty$, the
geometry is $\calD=\partial\Omega$, and the output is
$u(x)=C_p(x)$ for $x\in\partial\Omega$.

For the nonlifting solution considered here, following the source-panel formulation \cite{hess1967calculation}, we represent the potential as
\begin{equation}
\Phi(x) = v_{\infty} \cdot x + \int_{\partial\Omega} \eta(y) \kappa(x - y) \dd S_y,\quad \kappa(x-y)=\frac{1}{4\pi\lVert x-y\rVert_2},
\label{eq:panel_phi_decomposition}
\end{equation}
where $\eta$ is an unknown density on $\partial\Omega$.  $\eta$ is determined by enforcing the no-penetration condition in \eqref{eq:laplace_exterior}:
\begin{align}
\label{eq:potential-flow-bie}
v_\infty\cdot n_{x} - \frac{1}{2} \eta(x) + \int_{\partial\Omega} \eta(y) \nabla_x \kappa(x-y) \dd S_y\cdot n_{x}
= 0.
\end{align}
Equation \eqref{eq:potential-flow-bie} is a Fredholm integral equation of the second kind for  $\eta$. Once $\eta$ is obtained, the surface velocity is computed as 
\begin{align}
\label{eq:potential-flow-velocity}
v(x) = \nabla \Phi(x) 
=  v_\infty + \int_{\partial\Omega} \eta(y) \nabla_x \kappa(x-y) \dd S_y. 
\end{align}
Bernoulli's relation then gives
\begin{equation}
\label{eq:pressure-coefficient}
C_p(x) = \frac{p(x) - p_\infty}{\frac{1}{2}\rho\lVert v_{\infty} \rVert_2^2} = 1 - \frac{\lVert v \rVert_2^2 }{\lVert v_\infty \rVert_2^2}
 \qquad
    x\in\partial\Omega.
\end{equation} 
Consequently, solving \cref{eq:potential-flow-bie}, evaluating the layer potential
in \cref{eq:potential-flow-velocity}, and applying
\cref{eq:pressure-coefficient} define the aerodynamic operator in
\cref{eq:aerodynamic-map}.

From this perspective, learning the solution operator \eqref{eq:G-continuous-intro} on variable geometries can be viewed as learning a geometry-dependent and potentially singular kernel $\kappa$, together with its associated operations: kernel integration, the solution of Fredholm integral equations, and their compositions. Since a Fredholm integral equation of the second kind admits a Neumann series expansion~\cite[Section 2.4]{kress1989linear} in terms of repeated kernel integration, accurately learning the kernel integral operator is the central task.
Motivated by this view, we represent geometries $\calD$ as point clouds~\cite{qi2017pointnet,li2018pointcnn,zeng2025point} and design kernel-based neural operators that act directly on functions defined over $\calD$. The proposed neural operator is built by composing parameterized kernel integral operators on $\calD$ (i.e., \cref{eq:potential-flow-bie,eq:potential-flow-velocity}), implemented with efficient evaluation, with pointwise nonlinearities, thereby approximating the solution operator without requiring an explicit geometric parameterization.
This formulation naturally promotes generalization across diverse geometries, including settings with significant geometric and topological variation.

\subsection{Contributions}
We adopt a kernel integral perspective on neural operators to better understand and improve their ability to generalize across variable geometries. Specifically, our contributions include:

\begin{enumerate}
\item We establish approximation results for singular kernel integral operators via a multiscale decomposition inspired by Ewald summation, and we derive corresponding error bounds.

\item We introduce the multiscale point cloud neural operator (M-PCNO), a neural network-based surrogate for PDE-induced operators on variable geometries with linear inference cost in the number of points. We also release an educational software package at \url{https://github.com/PKU-CMEGroup/NeuralOperator}.

\item We demonstrate geometric generalization in representative numerical experiments across diverse geometries, including large-scale three-dimensional fluid dynamics examples.

\end{enumerate}

\subsection{Preliminaries and Literature Review}
This work aims to advance the understanding of how neural operators generalize when learning PDE solution operators on variable geometries. In this section, we review related neural operator approaches and classical fast kernel summation methods.
The latter provides guiding principles
for constructing linear operators that approximate singular kernel integrals efficiently.

\subsubsection{Neural Operator Approaches}
Neural operators \cite{zhu2018bayesian,khoo2019switchnet, lu2021learning,li2020fourier} are neural network-based surrogate models that approximate PDE solution operators such as \eqref{eq:G-continuous-intro}.
Their key feature is that they approximate mappings between function spaces at the continuous level, rather than solutions at a fixed discretization.
By separating the operator representation from the discretization used in computation, neural operators can generalize across mesh resolutions.
Existing architectures are typically built from integral operators \cite{li2020fourier,kovachki2023neural,gin2021deepgreen,boulle2022learning,cao2024laplace,hao2026multiscale}, often motivated by Green's functions, or from differential operators \cite{liu2024neural,zeng2025point}; in either case, a standard numerical discretization is applied afterward to obtain concrete algorithms.

Extending neural operators to variable computational domains remains challenging.
Common strategies either map each geometry to a fixed reference domain via a parametric deformation~\cite{li2023fourier,yin2024dimon,xiao2024learning}, or embed the domain into a fixed bounding box, for example, using zero padding or signed distance functions~\cite{he2024geom,ye2024pdeformer,duvall2025discretization,liu2024domain,li2024geometry}. Both approaches reduce the problem to a fixed representation and enable the use of standard neural operator architectures. However, deformation maps may not be well-defined for complex geometries~\cite{xiao2024learning}, and embedding-based approaches often rely on mesh-resolution–dependent interpolation or extrapolation.

An alternative strategy represents variable geometries $\calD$ as point clouds, which naturally encode detailed geometric information. Point cloud networks, such as PointNet \cite{qi2017pointnet} and its extensions \cite{li2018pointcnn}, enable the extraction of geometric features for tasks such as classification and segmentation. 
Operator learning frameworks, including DeepONet \cite{lu2021learning,hu2025manifold} and transformer-based models \cite{cao2021choose,junfengpositional,wu2024transolver,wen2025geometry,wang2025mno}, have also been extended to operate on point cloud representations, often using encoder-based architectures to reduce memory and computational costs.
When connectivity information among points is available, graph neural networks provide a natural framework for PDE surrogate modeling~\cite{pfaff2020learning,liu2024laflownet,gao2025generative,shen2025vortexnet}. Although standard graph-based methods can be sensitive to mesh resolution, mesh-independent aggregation can mitigate this dependence, leading to graph neural operators~\cite{li2020neural,li2020multipole,mousavi2025rigno}.
Moreover, point cloud representations can be augmented with additional geometric information to further improve accuracy and robustness~\cite{goswami2022deep,he2024geom, ye2024pdeformer,duvall2025discretization,serrano2023operator,quackenbush2025transferable}.
In this work, we adopt the point cloud strategy and design neural operators that act directly on geometries $\calD$ represented as point clouds. Guided by boundary element formulations such as  \cref{eq:potential-flow-bie,eq:potential-flow-velocity}, we study how these operators generalize under variations in domain geometry.

\subsubsection{Classical Fast Kernel Summation Methods}

Neural operators are often built from kernel integral operators, motivated by Green's functions.
In many PDE settings, the associated kernels are singular and depend explicitly on geometric quantities such as boundary normals. Moreover, efficient evaluation of these kernel integrals is crucial for deployment at scale. These issues have received comparatively less attention in the neural operator literature, but they are central to the present work. We therefore review classical fast kernel summation methods for efficiently evaluating singular kernel integrals, which motivate our neural operator design.

Methods such as the fast multipole method~\cite{barnes1986hierarchical,greengard1987fast,anderson1992implementation,cheng1999fast,ying2004kernel,fong2009black} and hierarchical matrix techniques \cite{hackbusch2000sparse,bebendorf2000approximation,borm2003introduction} accelerate dense kernel matrix–vector multiplications by exploiting scale separation: far-field interactions are approximated using low-rank representations, while near-field interactions involving potentially singular kernels are evaluated exactly. As a result, these methods can achieve linear or quasi-linear computational complexity. 
Several works have explored neural operators inspired by such multiscale ideas~\cite{fan2019multiscale2,fan2019multiscale,boulle2023elliptic,sun2025learning}. 
However, incorporating these fast summation methods into learning-based neural operators for general kernels, especially in a manner optimized for modern GPU architectures, remains challenging.
In this work, we revisit an earlier strategy originating with Ewald summation \cite{ewald1921berechnung,darden1993particle,hockney2021computer}, in which far-field interactions are approximated in the Fourier domain \cite{bleszynski1996aim,phillips2002precorrected,strain1992fast}, while local interactions are computed exactly or approximated via local Taylor expansions \cite{greengard1990fast}. We adopt this strategy to the design of neural operators and demonstrate its effectiveness across a broad class of kernels and related problems, achieving moderate accuracy while naturally enabling efficient GPU implementations.

\subsection{Organization}
\label{ssec:organization}
In \cref{sec:u-approx-kernels}, we study Ewald–type decompositions for approximating singular kernel integral operators and provide a theoretical analysis of the resulting approximation errors. Building on these results, \cref{neural_layer} develops neural layers for efficient kernel integration across a broad class of commonly encountered kernels.
\Cref{sec:mpcno} introduces the M-PCNO. Numerical experiments in \cref{sec:numerics} validate the theory and demonstrate geometric generalization. Finally, \cref{sec:conclusion} provides concluding remarks.

\section{Approximation of Singular Kernel Integral Operators}
\label{sec:u-approx-kernels}
In this section, we consider a family of variable computational domains $\Omega$. We assume that every domain is contained in the box 
$$B = [0,l_1]\times[0,l_2]\times\cdots\times[0,l_d],$$
where $d$ is the spatial dimension and $\{l_i\}$ is the side lengths in the $i$th coordinate direction. We further assume a uniform separation from the boundary: there exists $d_B > 0$ such that $\mathrm{dist}(\Omega, \partial B) \geq d_B$ for all domains under consideration. 

Many physical problems of practical interest involve layer potential operators acting on fields defined on $\partial \Omega$:
\begin{subequations}
\label{eq:layer-potential}
    \begin{align}
         &\textrm{Single layer potential}: (\calK f)(x) = \int_{\partial \Omega} \kappa(x-y) f(y) \dd S_y, \label{eq:single-layer-potential}\\
    &\textrm{Double layer potential}: (\calK f)(x) = \int_{\partial \Omega} \frac{\partial \kappa(x-y)}{\partial n_y} f(y) \dd S_y,\label{eq:double-layer-potential}\\
    &\textrm{Adjoint double layer potential}: (\calK f)(x) = \int_{\partial \Omega} \frac{\partial \kappa(x-y)}{\partial n_x} f(y) \dd S_y\label{eq:adjoint-double-layer-potential},
    \end{align}
\end{subequations}
where $x\in\partial\Omega$, $\dd S_y$ denotes the boundary measure, and $\kappa$ is a translation-invariant kernel.
A canonical example is the fundamental solution of the Laplace operator:
\begin{equation}
    \kappa(x-y) = \begin{cases}
        -\frac{1}{2\pi}\ln \lVert x - y \rVert_2  \qquad  & d = 2\\
        \frac{1}{(d-2) \alpha_d \lVert x - y\rVert_2^{d-2}}  \qquad & d > 2,
    \end{cases}
\end{equation}
where $\alpha_d$ denotes the surface area of the unit sphere in $\bR^d$.

A central objective of this work is to approximate such layer potential operators in \cref{eq:layer-potential} by learning a
representation of the kernel $\kappa$ that is shared across geometries and enables efficient
evaluation of the associated integrals.
Since $\partial\Omega \subset B$, all differences $x-y$ lie in 
\begin{align}
\label{eq:B_2}
  B_2 = [-l_1,l_1]\times[-l_2,l_2]\times\cdots\times[-l_d,l_d].  
\end{align}
Accordingly, we focus on learning $\kappa: B_2 \rightarrow \bR$.
We further assume that $\kappa$ is periodic on $B_2$.
When the original kernel is not periodic, Fourier continuation~\cite{bruno2001fast} can be used to construct a periodic extension on $B_2$ and smooth near $\partial B_2$. This is justified because the geometries of interest remain at least a distance $d_B$ away from the boundary, so the periodic extension does not affect kernel evaluations on the geometries considered.

Because the kernel $\kappa$ is generally singular at the origin, efficient evaluation of the integrals in \cref{eq:layer-potential}  often benefits from a multiscale decomposition that separates long- and short-range interactions. As a canonical example, consider the three-dimensional Coulomb kernel $\kappa(x-y) = \frac{1}{r}$ with $r = \lVert x - y\rVert_2$, Ewald summation~\cite{ewald1921berechnung} employs the decomposition 
\begin{align}
\label{eq:ewald}
\frac{1}{r} = \frac{{\rm erf}(r/\delta)}{r} + \frac{{\rm erfc}(r/\delta)}{r}, \quad \textrm{where} \quad
{\rm erf}(r) = \frac{2}{\sqrt{\pi}}\int_0^{r} e^{-t^2}\dd t,  \quad  {\rm erfc}(r) = 1 - {\rm erf}(r).
\end{align}
Here $\delta > 0$ is the Ewald splitting parameter that separates the long- and short-range contributions.
The first term is a smooth long-range component~(including at $r=0$) and can be handled efficiently in Fourier space. The second term is a localized short-range component that decays rapidly for $r \gtrsim \delta$; indeed, $\frac{{\rm erfc}(r/\delta)}{r} \leq \frac{\delta}{\sqrt{\pi}r^2} e^{-r^2/\delta^2}$.
Consequently, the short-range contribution can be evaluated directly, or approximated locally, within a neighborhood of radius $\mathcal{O}(\delta)$.

Motivated by the Ewald splitting \eqref{eq:ewald}, we introduce an Ewald-type decomposition for approximating general singular kernel integral operators. This decomposition supports efficient evaluation and yields polynomially decaying approximation error bounds, summarized in the following theorem. The proof is deferred to \cref{sec:proof-theorem:approximation-error}.

\begin{theorem}
\label{theorem:approximation-error}
Let $d\geq2$ and define $B:=\left[0,\frac12\right]^d$, $B_2:=\left[-\frac12,\frac12\right]^d$.
Let $\partial\Omega\subset B$ be the boundary geometry. Since
$x-y\in B_2$ for every $x,y\in\partial\Omega$, define the
linear integral operator
\begin{equation}
\label{eq:kernel-operator-theorem}
    (\calK f)(x)=\int_{\partial\Omega}\kappa(x-y)f(y)\,dS_y\qquad x\in\partial\Omega.
\end{equation}

For $\delta\in(0,\tfrac12)$, let $\rho_\delta(y) = \frac{1}{(2\pi \delta^2)^ {d/2}} e^{-\frac{ \lVert y \rVert_2^2 }{2 \delta^2}}$ denote the Gaussian mollifier.
Periodically extend $\kappa$ from $B_2$ to
$\mathbb{R}^d$ and define
\[
(\kappa * \rho_\delta)(x)
:=
\int_{\mathbb{R}^d}
\kappa(x-y)\rho_\delta(y)\,\mathrm{d}y.
\]
We introduce the decomposition
\begin{equation}
\kappa
=
\kappa_{\rm long}+\kappa_{\rm short},
\qquad
\kappa_{\rm long}:=\kappa*\rho_\delta,
\qquad
\kappa_{\rm short}:=\kappa-\kappa*\rho_\delta.
\end{equation}
The corresponding
operators are
\begin{align*}
    (\calK_{\rm long}f)(x)=\int_{\partial\Omega}\kappa_{\rm long}(x-y)f(y)\,\dd S_y  \quad 
    (\calK_{\rm short}f)(x)=\int_{\partial\Omega}\kappa_{\rm short}(x-y)f(y)\,\dd S_y.
\end{align*}
For $p\in\mathbb N$, let
$\widehat{\kappa_{\rm long}}_k$ denote the Fourier coefficients of
$\kappa_{\rm long}$ and define the truncated Fourier representation and the corresponding truncated long-range operator:
\begin{align*}
    \hat{\kappa}_{\rm long} =\sum_{ k:\lVert k\rVert_\infty \leq  p }\widehat{\kappa_{\rm long}}_k e^{2\pi i k \cdot x}\qquad
    (\widehat{\calK}_{\rm long}f)(x)=\int_{\partial\Omega}\hat{\kappa}_{\rm long}(x-y)f(y)\,\dd S_y.
\end{align*}
For $\epsilon>\delta$, let $B_\epsilon(x):=\left\{y\in\mathbb{R}^d:\lVert y-x\rVert_2\leq\epsilon\right\}$ and define the localized short-range operator 
\begin{equation}
\label{eq:localized-short-range-operator}
    (\calK_{\rm short}^{\epsilon}f)(x)
    :=
    \int_{\partial\Omega\cap B_\epsilon(x)}
    \kappa_{\rm short}(x-y)f(y)\,\dd S_y,
    \qquad x\in\partial\Omega.
\end{equation}

Assume the following.
\begin{enumerate}[
    label=\textnormal{(A\arabic*)},
    ref=\textnormal{(A\arabic*)},
    leftmargin=*
]
    \item\label{assumption:kappa_assumption}
    \textbf{Kernel regularity.} 
    The kernel $\kappa$ is periodic and translation-invariant, with $\kappa \in L^1(B_2) \cap C^2(B_2\setminus\{0\})$. Moreover, there exists a constant $C_\kappa>0$ such that $
     \lVert \nabla^j \kappa(x) \rVert_2 \leq\frac{C_\kappa}{\lVert x \rVert_2^{j+d-1}},\quad j=0,1,2$,
 for all $x\in B_2\setminus\{0\}$.
    \item\label{assumption:geometry_assumption}
    \textbf{Geometric regularity.}
    The boundary geometry  $\partial \Omega \subset B$, is a compact embedded
    $(d-1)$-dimensional Lipschitz hypersurface with Lipschitz character bounded by $C_L$ in the sense of \cite[Definition 1.11]{rataj2019curvature}.
    \item\label{assumption:input_assumption}
    \textbf{Input regularity.}
    The input satisfies $f\in W^{2,\infty}(\partial \Omega)$. 

    \item\label{assumption:short-range_assumption}
    \textbf{Local short-range approximability.}
    There exist $\epsilon_0>0$, a local approximation order $q\geq0$, and a
constant $C_{\rm short}(C_\kappa, C_L,d)>0$, such that, for every $0<\epsilon\leq\epsilon_0$, there exists a linear local operator
$\widehat{\calK}_{\rm short}^{\epsilon}$,
\begin{equation}
\label{eq:local-short-range-error}
    \left\|
        \calK_{\rm short}^{\epsilon}f
        -
        \widehat{\calK}^{\epsilon}_{\rm short}f
    \right\|_{L^\infty(\partial\Omega)}
    \leq
    C_{\rm short} 
    \lVert f\rVert_{W^{2,\infty}(\partial\Omega)} \epsilon^{d+q}.
\end{equation}
\end{enumerate}
 
Then the following estimates hold:
    \begin{enumerate}
    [
    label=\textnormal{(P\arabic*)},
    ref=\textnormal{(P\arabic*)},
    leftmargin=*
]
    \item
    \textbf{Long-range Fourier approximation.}
    Define $C_{\rm long}=\frac{d\lVert\kappa\rVert_{L^1(B_2)}}{2\pi^2}(1+\frac{1}{\sqrt{2\pi}})^{d-1}>0$, then
        \begin{equation}
        \label{eq:kappa_long_conclusion}
            \Bigl\lVert \widehat{\calK}_{\rm long} f  - \calK_{\rm long} f \Bigr\rVert_{L^\infty(\partial\Omega)} \leq C_{\rm long}\frac{e^{-2\pi^2\delta^2p^2}}{\delta^{d+1} p}\lVert f\rVert_{L^1(\partial\Omega)}.
        \end{equation}

\item
    \textbf{Short-range localization.}
    For any $x\in B_2$ satisfying $\lVert x\rVert_2\geq\delta$, 
    \begin{equation}
    \label{eq:kappa_short_conclusion}
        \Bigl|\kappa_{\rm short}(x)\Bigl|\leq\frac{dC_{\kappa}}{2(1-\frac{1}{\sqrt d})^{d+1}}\frac{\delta^2}{\lVert x\rVert_2^{d+1}}+c_{\rm short}\frac{e^{-\frac{1}{2d}\frac{\lVert x\rVert_2^2}{\delta^2}}}{\delta^{d}},
    \end{equation}
   where  $c_{\rm short} = 
\frac{\lVert\kappa\rVert_{L^1(B_2)}}{(2\pi)^{d/2}}(1+\frac{2^dd}{(1-e^{-1/2})^d})+\frac{C_{\kappa}C_{d/2}}{\Gamma(\frac{d}{2})}
    \bigl(
    \frac{1}{2^{d/2}}+d
    \bigr)$.
    Here $\Gamma$ denotes the Gamma function, and $C_{d/2}$ is the Gamma-function–related constant defined in \cref{lemma:Gamma_function}.

    \item
    \textbf{Combined operator approximation.}
     Define 
\begin{equation}
\label{eq:kappa_tilde}
    \widehat{\calK} = \widehat{\calK}_{\rm long}+\widehat{\calK}^{\epsilon}_{\rm short}.
\end{equation}
Fix $\gamma \in (0,1)$
and set $\delta = p^{-\gamma}$, $\epsilon = \delta^t = p^{-\gamma t}$, and $t = \frac{2}{q+d+2}$.
For every integer $p$ satisfying $p > \max\{2^{\frac{1}{\gamma}}, \epsilon_0^{-\frac{1}{\gamma t}}\}$, we obtain the $L^\infty$ error estimate
\begin{align}
        \lVert \widehat{\calK} f - \calK f\rVert_{L^\infty(\partial\Omega)} 
&\leq C\Bigl( p^{-1+\gamma(d+1)}e^{-2\pi^2p^{2-2\gamma}} + p^{d\gamma}e^{-\frac{1}{2d}p^{\gamma(2-2t)}} +2p^{-\gamma t(q+d)}\Bigl) \nonumber\\
&=\calO\Bigl(p^{-\gamma \bigl(1 + \frac{q + d - 2}{q+d+2}\bigr)}\Bigr), \label{eq:calK_error}  
\end{align}
where $C > 0$ depends on  $C_\kappa$, $\lVert \kappa \rVert_{L^1(B_2)}$, $C_L$,  $d$, and $\lVert f\rVert_{W^{2,\infty}(\partial\Omega)}$.
\end{enumerate}
\end{theorem}

\begin{remark}[Convergence rate]
\Cref{eq:kappa_tilde} approximates the linear integral operator $\calK$ by combining a truncated Fourier representation of the smooth long-range operator with a localized approximation of the short-range operator. 
The error bound \cref{eq:calK_error}
consists of two exponentially decaying terms and one polynomially decaying term. For sufficiently large $p$, the polynomial term $p^{-\gamma t(q+d)} = p^{-\gamma \bigl(1 + \frac{q + d - 2}{q+d+2}\bigr)}$ dominates. 
 Since $\gamma\in(0,1)$ is a free parameter, we may take $\gamma\to 1$, in which case the polynomial decay rate approaches $p^{-\bigl(1 + \frac{q +d- 2}{q+d+2}\bigr)}.$
This rate is strictly faster than $\calO(p^{-1})$ when the local approximation order satisfies $q > 0$, and it can be made arbitrarily close to $\calO(p^{-1})$ when $q=0$. The theorem thus clarifies how the truncated mode number $p$ (i.e., retaining modes with $\lVert k \rVert_{\infty}\leq p$) controls the approximation error, a relationship that is also observed numerically in \cref{sec:numerics}.
\end{remark}

\begin{remark}[Applicability of the assumptions]
    The kernel regularity assumptions in \ref{assumption:kappa_assumption} are satisfied by many commonly used kernels, including those listed in \cref{tab:kernel-short-range-approximation}. 
    The regularity assumption on $f$ is used to control the Taylor remainder
in the local approximation
\cref{eq:local-short-range-error}.
The curvature-dependent expansions   in
\cref{tab:kernel-short-range-approximation} additionally require sufficient
smoothness of the boundary geometry.
Under these additional geometric
assumptions, the displayed expansions support the conservative choice
$q=0$ for the local short-range approximation in \ref{assumption:short-range_assumption}. Higher local approximation
orders are available for some two-dimensional double-layer-type operators.
\end{remark}

\begin{remark}[Implications for neural layer design]
    This decomposition, combining a truncated Fourier representation with a localized approximation, serves as a useful conceptual template for our neural layer design. 
    In classical Ewald-type decompositions, selecting the splitting parameter $\delta$ is not straightforward. In contrast, our neural layers explicitly include Fourier and local components and learn how to weight and combine them from data, without committing to the specific construction \cref{eq:kappa_tilde}.
    \Cref{theorem:approximation-error} can therefore be viewed as an achievability benchmark: it shows that this multiscale ansatz can attain an error no larger than \eqref{eq:calK_error} when training is effective.
    For simplicity, the theorem is stated on $B= [0,\frac{1}{2}]^d$; the corresponding estimates for a general box 
$B=\prod_{i=1}^{d}[0,l_i]$ follow by rescaling. In the neural operator design below, we allow arbitrary bounding boxes.
\end{remark}

\section{Neural Layers for Kernel Integrals}
\label{neural_layer}
In this section, we design neural layers that approximate singular kernel integral operators guided by the Ewald-type decomposition in \cref{theorem:approximation-error}. Specifically, we approximate the smooth long-range component using a truncated Fourier representation and treat the short-range component via localized Taylor expansions. 

We consider integral operators defined on a sufficiently smooth boundary geometry $\calD=\partial\Omega\subset\mathbb{R}^d$.
Let $\mathcal{K}$ be an integral operator acting on sufficiently smooth $\bR^{d_g}$-valued functions on $\calD$, $\mathcal{K}:  \{g: \calD \rightarrow \bR^{d_{g}}\} \longrightarrow \{g: \calD \rightarrow \bR^{d_{g}}\}$.
The operator is defined by
\begin{equation}
\label{eq:kernel-map}
 (\mathcal{K}g)(x) = \int_{\calD} \kappa(x-y; n_x, n_y) g(y) \dd S_y, \quad x \in \calD.
\end{equation}
where $\kappa$ is translation-invariant in the displacement $x-y$ and may have a singularity at the origin. It may also depend on the outward unit normals $n_x$ and $n_y$.
For the kernels considered in this work, the normal dependence enters multiplicatively, encompassing the examples listed in
\cref{tab:kernel-short-range-approximation}.

We first describe the approximation of the smooth long-range component and then develop a local approximation of the short-range component, and finally combine these ingredients to construct the proposed multiscale point cloud neural layer.

\subsection{Long-Range Approximation}
\label{ssec:long-range}
For the long-range component, we approximate the smooth part of the kernel-induced map in \eqref{eq:kernel-map} by a truncated Fourier representation:
\begin{equation}
\begin{split}
   \int_{\calD}\kappa_{\rm long}(x-y; n_x, n_y) g(y) \dd S_y  
   &\approx  (\bK_{\rm long} g)(x). 
\end{split}
\end{equation}
Following  Fourier neural layer parameterizations~\cite{nelsen2021random,li2020fourier,kovachki2023neural,huang2024operator,de2022cost}, we represent the translation-invariant component of the smooth long-range kernel by the truncated Fourier
expansion
$$\kappa_{\rm long}(x - y) \approx \sum_{k:\lVert k\rVert_\infty \leq  p } e^{2\pi i k \cdot\frac{x - y}{2l}} W_{v}^{k},$$
where the learnable matrices $W_{v}^{k}$ parameterize the Fourier coefficients of the matrix-valued kernel. The expansion is truncated to modes satisfying $\|k\|_\infty\le p$. The normalized displacement is defined componentwise by 
\begin{equation}
    \frac{x - y}{2l} := (\frac{x_1 - y_1}{2l_1},\cdots,\frac{x_d - y_d}{2l_d}),
\end{equation} 
where $\{l_i\}_{i=1}^d$ are determined by the bounding box $B_2$ defined in \cref{eq:B_2}. 
Since the kernel may depend on both the source normal $n_y$ and the target normal $n_x$, we incorporate this dependence through the following factorized representation:
\begin{equation}
\label{eq:K_long_surf}
\begin{split}
(\bK_{\rm long}^{(1)} g)(x)
&:=  \sum_{k:\lVert k\rVert_\infty \leq  p}  \int_{\partial \Omega} e^{2\pi i k \cdot\frac{x - y}{2l}} W_{v}^{k}  W_1 \begin{bmatrix}
    g(y)  \\ 
    \textrm{vec}\bigl(g(y) \otimes n_y \bigr)
\end{bmatrix} \dd S_y,  \\
(\bK_{\rm long} g)(x)
&:=  W_2 
\begin{bmatrix}
\bK_{\rm long}^{(1)} g(x)  \\
\textrm{vec}\bigl((\bK_{\rm long}^{(1)} g)(x) \otimes n_x \bigr)
\end{bmatrix}.
\end{split}
\end{equation}
Here
$W_1\in\bR^{d_g\times d_g(d+1)}$ and
$W_2\in\bR^{d_g\times d_g(d+1)}$ are learnable matrices.
The matrix $W_1$
encodes dependence on the source normal $n_y$ by combining $g(y)$ with
$\operatorname{vec}(g(y)\otimes n_y)$, as occurs in the double layer
potential in \eqref{eq:double-layer-potential}. Similarly,
$W_2$ combines the normal-independent features with features depending on
the target normal $n_x$, as in the adjoint double layer potential in
\eqref{eq:adjoint-double-layer-potential}.

\subsection{Short-Range Approximation}
\label{ssec:short-range}
The short-range component is designed to capture the near-field contribution
to the kernel integral that is not accurately resolved by the truncated
Fourier approximation in \cref{eq:K_long_surf}, particularly near the kernel singularity. 
For $x\in\calD$, let
\begin{equation}
\label{eq:B_epsilon}
    B_\epsilon(x)
    :=
    \left\{
        y\in\mathbb{R}^d:
        \lVert y-x\rVert_2\leq\epsilon
    \right\}
\end{equation}
denote the Euclidean ball of radius $\epsilon$ centered at $x$.
We approximate the corresponding local contribution by
\begin{equation}
\label{eq:local-integral}
\int_{\calD\cap B_{\epsilon}(x)} \kappa'(x-y, n_x, n _y) g(y) \dd S_y \approx  (\bK_{\rm short} g)(x),
\end{equation}
where $\kappa'$ denotes the residual short-range kernel after subtracting the long-range Fourier component from $\kappa$.

To approximate \cref{eq:local-integral}, we apply an intrinsic Taylor expansion of $g$ about $x$.
For a sufficiently smooth function $g$,
\begin{equation}
\label{eq:intrinsic-taylor-short}
g(y) = g(x) +  \nabla_{\calD} g(x)   v_x(y)   + \mathcal{O}(\lVert v_x(y) \rVert_2^2),
\end{equation}
where $v_x(y):= \log_x(y) \in T_x\calD$ is the tangent space coordinate of $y$ at $x$ induced by the logarithmic map.  
Because $y\in \calD\cap B_\epsilon(x)$, we have $\lVert v_x(y)\rVert_2 = \calO(\epsilon)$  provided that $\calD$ is regular and $\epsilon$ is sufficiently small.
Since $\calD =\partial \Omega$ is a curve or surface, $\nabla_{\calD}$ denotes the tangential gradient, obtained by projecting the ambient gradient onto the
tangent space $T_x\calD$.
Substituting \cref{eq:intrinsic-taylor-short}  into  \cref{eq:local-integral} yields
\begin{equation}
\label{eq:local-integral-expand}
\int_{\calD\cap B_\epsilon(x)} \kappa'(x-y; n_x, n_y) g(y) \dd S_y = M_0(x) g(x) +  M_1(x) : \nabla_{\calD} g(x)  + \mathcal{R}(x),
\end{equation}
where $:$ denotes the natural contraction between the first-moment tensor
$M_1(x)$ and the tangential Jacobian $\nabla_{\calD}g(x)$. The local kernel moments are defined by
\begin{align*}
M_0(x) = \int_{\calD\cap B_\epsilon(x)} \kappa'(x-y; n_x, n_y) \dd S_y, \,
M_1(x) = \int_{\calD\cap B_\epsilon(x)} \kappa'(x-y; n_x, n_y) v_x(y) \dd S_y. 
\end{align*}
The remainder $\mathcal{R}(x)$ arises from integrating the Taylor remainder $\mathcal{O}(\lVert v_x(y) \rVert_2^2)$ against the residual kernel. 

The construction \eqref{eq:local-integral-expand} does not approximate the singular
kernel pointwise. Instead, it approximates the local contribution of the kernel
after integration. For the weakly singular kernels considered here and listed in \cref{tab:kernel-short-range-approximation}, the resulting local moments are finite, thereby mitigating the effect of the kernel's pointwise singularity.

\begin{table}[htbp]
    \begin{center}
    \resizebox{\textwidth}{!}{%
        \begin{tabular}{ccc}
        \Xhline{1.2pt}
        Potential  & Kernel  & Short-range asymptotic approximation   \\ \Xhline{1.2pt}
            \makecell{2D Laplacian single \\layer potential}  & $\frac{-1}{2\pi}\log{\lVert x-y\rVert_2}$  & $-\frac{\epsilon\log\epsilon-\epsilon}{\pi}f(x) + o(\epsilon^2)$   \\ \hline
            \makecell{2D Laplacian double \\layer potential}  & $\frac{(x-y)\cdot n_y}{2\pi\lVert x - y \rVert^2_2}$  & $-\frac{\epsilon \mathrm{tr}[\nabla_{\calD} n_x]}{2\pi}f(x) + \calO(\epsilon^3)$   \\ \hline
            \makecell{2D Modified Laplacian \\double layer potential}  & $\frac{x - y}{2\pi \|x - y\|_2^2}$  & $ 
 -\frac{\epsilon \mathrm{tr}[\nabla_{\calD} n_x] n_x}{2\pi}f(x)  -\frac{\epsilon}{\pi}\nabla_{\cal D} f(x)
+ \mathcal{O}(\epsilon^3)$    \\ \hline
            \makecell{2D Adjoint Laplacian \\double layer potential} &  $\frac{(y-x)\cdot n_x}{2\pi\lVert x - y \rVert^2_2}$  & $-\frac{\epsilon \mathrm{tr}[\nabla_{\calD} n_x]}{2\pi}f(x) +  \calO(\epsilon^3)$    \\ \hline
            \makecell{2D Stokeslet} &$\frac{1}{4\pi}\bigl(-\log{\lVert x-y\rVert_2}I_2+\frac{(x-y)(x-y)^T}{\lVert x-y\rVert_2^2}\bigr)$ & $\frac{-\epsilon\log\epsilon+\epsilon}{2\pi}f(x) + \frac{\epsilon}{2\pi}\bigl(I_2 - n_x  n_x^T\bigr)f(x)+o(\epsilon^2)$   \\ 
            \Xhline{1.2pt}
            \makecell{3D Laplacian single \\layer potential}  & $\frac{1}{4\pi \lVert x-y\rVert_2} $  & $\frac{\epsilon}{2}f(x) + \mathcal{O}(\epsilon^3)$ \\ \hline
            \makecell{3D Laplacian double \\layer potential} & $\frac{(x-y)\cdot n_y}{4\pi\lVert x - y \rVert^3_2}$  & $-\frac{\epsilon \mathrm{tr}[\nabla_{\calD} n_x]}{8}f(x)+\mathcal{O}(\epsilon^3)$\\ \hline
            \makecell{3D Modified Laplacian \\double layer potential}   & $\frac{x - y}{4\pi \|x - y\|_2^3}$  & $ 
  -\frac{\epsilon \mathrm{tr}[\nabla_{\calD} n_x] n_x }{8} f(x)  - \frac{\epsilon}{4}\nabla_{\calD} f(x) 
+ \mathcal{O}(\epsilon^3)$ \\ \hline
            \makecell{3D Adjoint Laplacian \\double layer potential} & $\frac{(y-x)\cdot n_x}{4\pi\lVert x - y \rVert^3_2}$  & $-\frac{\epsilon \mathrm{tr}[\nabla_{\calD} n_x]}{8} f(x)+\mathcal{O}(\epsilon^3)$ \\ \hline
            \makecell{3D Stokeslet}   & $\frac{1}{8\pi}\bigl(\frac{1}{\lVert x-y\rVert_2}I_3 + \frac{(x-y)(x-y)^T}{\lVert x-y\rVert_2^3}\bigr)$  &  $\frac{\epsilon}{4}f(x) + \frac{\epsilon}{8} (I - n_x n_x^T) f(x) +\mathcal{O}(\epsilon^3)$ \\ 
            \Xhline{1.2pt}
        \end{tabular}
        }
    \end{center}
    \caption{Summary of common 2D and 3D layer potentials, their kernels, and short-range asymptotic approximations of $\displaystyle \int_{\partial\Omega\cap B_{\epsilon}(x)} \kappa(x-y; n_x, n _y) f(y) \dd S_y$. When the kernel $\kappa$ is singular, the integral is understood in the sense of the Cauchy principal value. The derivations are deferred to the supplementary material.}
    \label{tab:kernel-short-range-approximation}
\end{table}

These local kernel moments, $M_0$ and $M_1$, encode both the kernel profile and the local geometry. To illustrate their structure, we consider several commonly used 2D and 3D layer potentials. The corresponding short-range asymptotic expansions are summarized in \cref{tab:kernel-short-range-approximation}. In each case,  \cref{eq:local-integral-expand} captures the leading-order behavior.
The remainders are $o(\epsilon^2)$ for one-dimensional boundaries and
$\mathcal{O}(\epsilon^3)$ for two-dimensional boundaries.
For the kernels listed in \cref{tab:kernel-short-range-approximation}, the leading-order moments $M_0$ and $M_1$ depend
on the local geometry through low-order combinations of the outward normal
$n_x$ and the curvature information encoded by $\nabla_{\calD}n_x$.
Motivated by this structure, we define
\begin{equation}
\label{eq:K_short_surf}
\begin{split}
&(\bK_{\rm short}^{(1)} g)(x)  := Wg(x) + b +  W_{g^{'},1}\texttt{SoftSign}\bigl(W_{g,1}\textrm{vec}\bigl(\nabla_{\calD}g (x)\bigr)\bigr) \\
&(\bK_{\rm short} g)(x) := (\bK_{\rm short}^{(1)} g)(x) + W_{g,4}\Bigl(\texttt{SoftSign}\Bigl(W_{g,3}
    \begin{bmatrix}
     n_x \\  \textrm{vec}\bigl(\nabla_{\calD}  n_x\bigr) 
    \end{bmatrix} \Bigr) \odot W_{g,2} g (x)\Bigr), 
\end{split}
\end{equation}
where $\odot$ denotes the componentwise product. 
The auxiliary term $ \bK_{\rm short}^{(1)} g $ captures the
dependence on the local value and tangential gradient of $g$.
Here $W\in\bR^{d_g\times d_g}$ and $b\in\bR^{d_g}$ define a pointwise affine map analogous to that used in the Fourier neural operator architecture~\cite{li2020fourier}. The matrices $W_{g,1}\in\bR^{d_g\times(d_g d)}$ and $W_{g',1}\in\bR^{d_g\times d_g}$ parameterize the correction associated with the tangential gradient of $g$.
Because the gradient features may have large magnitudes in regions of
sharp spatial variation, we apply the \texttt{SoftSign} function componentwise to obtain a bounded transformation. For a vector $z$, we have $[\texttt{SoftSign}(z)]_j = \frac{z_j}{1 + |z_j|}$.
This transformation bounds the magnitude of each component while preserving
its sign.
The second
term in \cref{eq:K_short_surf} constructs a learned embedding of the local geometry from $n_x$ and
$\nabla_{\calD}n_x$ and uses this embedding
to modulate the transformed feature $W_{g,2}g(x)$. The corresponding
learnable matrices are
$W_{g,2}\in\bR^{d_g\times d_g}$, $W_{g,3}\in\bR^{d_g \times (d^2+d)}$, and $ W_{g,4}\in\bR^{d_g\times d_g}$.

\subsection{Multiscale Point Cloud Neural Layer}
We combine the long-range operator $\bK_{\rm long}$ from \cref{ssec:long-range} with the short-range operator $\bK_{\rm short}$ from \cref{ssec:short-range} to construct a multiscale point cloud neural layer, denoted by $\calL^\calD$. Following a residual architecture
\cite{he2016deep},  the layer maps an input function $g: \calD \rightarrow \bR^{d_g}$ to an output function $\calL^\calD g: \calD \rightarrow \bR^{d_g}$ via
\begin{equation}
\begin{split}
\label{eq:mpcno-layer}
    &(\calL^\calD g)(x) = 
    g(x) + \sigma \Bigl( (\bK_{\rm long} g)\,(x)  + (\bK_{\rm short} g)\,(x)\Bigr),
\end{split}
\end{equation}
where $\sigma$ is a pointwise activation function. Throughout this work, we
use the Gaussian error linear unit (GELU)~\cite{hendrycks2016gaussian}.

The relative contributions of the long- and short-range components are not
prescribed a priori; instead, they are determined during training through
the parameters of $\bK_{\rm long}$ and $\bK_{\rm short}$. The radius $\epsilon$ in \cref{eq:local-integral} is introduced to motivate
the near-field expansion and is not an explicit parameter of the neural
layer.

\section{Multiscale Point Cloud Neural Operator}
\label{sec:mpcno}
In this section, we assemble the neural layers introduced in
\cref{neural_layer} into the multiscale point cloud neural operator
(M-PCNO). We first present the network architecture, then describe its
point cloud implementation and computational cost, and finally specify the
supervised training objective.

\subsection{Network Architecture}
The M-PCNO, denoted by $\calG_\theta$, is designed to approximate the
PDE-induced operator
\begin{align}
    \calG^{\dagger} : (f, \calD) \mapsto u,
\end{align}
where $f:\calD\rightarrow\bR^{d_f}$ and
$u:\calD\rightarrow\bR^{d_u}$ denote the prescribed input field and the
corresponding output field, respectively. Both fields are defined on the
boundary geometry $\calD=\partial\Omega$.  To encode geometric information, we augment $f$ with the coordinate
field $x \mapsto x$ and the outward
unit normal field $n_{(\cdot)}:\partial\Omega\rightarrow\bR^d, \quad x\mapsto n_x$.
We denote the resulting augmented input field by $\tilde{f}$.

The M-PCNO first applies a lifting map $\calP$ to embed the augmented input \(\tilde{f}\) 
into a higher dimensional latent feature field \(g_0:\calD\to\mathbb R^{d_g}\). It then applies $L$ multiscale point cloud
neural layers $\{\calL_i^\calD\}_{i=1}^{L}$, defined in
\cref{eq:mpcno-layer}, and finally maps the resulting latent feature field
to the output space through a projection map $\calQ$.
The architecture is given by
\begin{equation}
\label{eq:NO_architecture}
\begin{split}
&g_0(x) = \calP\bigl(\tilde{f}(x)\bigr), \qquad
g_{i} = \calL_i^\calD(g_{i-1}), \quad i=1,\dots,L, \qquad
u(x)=\calQ\bigl(g_L(x)\bigr),\\
&\calG_\theta(f, \calD) = \calQ\circ \calL^\calD_L \circ \calL^\calD_{L-1} \circ \cdots \calL^\calD_1 \circ \calP (\tilde{f}),
\end{split}
\end{equation}
where $\theta$ denotes the collection of all trainable parameters.

\subsection{Point Cloud Implementation}
At the discrete level, the geometry $\calD$ is represented by
a point cloud 
$$X = \{x^{(i)}\}_{i=1}^{N} \subset \calD,$$ 
together with mesh connectivity information and quadrature weights.
The lifting map $\calP$ is implemented as a pointwise affine map. 
For the long-range operator $\bK_{\rm long}$ in
\cref{eq:K_long_surf}, the integrals are evaluated on the point cloud using numerical quadrature \cite{lingsch2023beyond,zeng2025point}. 
For example, a representative
Fourier component is approximated by
\begin{equation}
\begin{split}
    \sum_{k:\lVert k\rVert_\infty \leq  p }  \int_{\calD} e^{2\pi i k \cdot \frac{x - y}{2l}} W_{v}^{k}  g (y) \dd S_y 
    &\approx 
    \sum_{k:\lVert k\rVert_\infty \leq  p } e^{2\pi i k \cdot \frac{x}{2l}}   W_{v}^{k} \Bigl(\sum_{i=1}^{N}   e^{-2\pi i k \cdot \frac{y^{(i)}}{2l}}    g (y^{(i)}) \dd S^{(i)} \Bigr),
\end{split}
\end{equation}
where $\dd S^{(i)}$ denotes the quadrature weight associated with the point $y^{(i)}$. 
The weights approximate the local arc-length or surface measure and are precomputed from the point cloud or an associated mesh.
For the short-range operator $\bK_{\rm short}$ in \cref{eq:K_short_surf}, the tangential gradients are approximated using a local least-squares reconstruction~\cite[Section 3.2]{zeng2025point}.
Specifically, $\nabla_{\calD} g(x^{(i)})$ is estimated from the function values at $x^{(i)}$ and a local neighborhood of nearby points; a neighborhood size on the order of $d$ typically suffices to make the local least-squares system well conditioned. Neighbor sets are precomputed from the point cloud or an associated connectivity graph. The gradient estimation can be implemented as a single message-passing step.
The projection map $\calQ$ converts the final latent representation into $d_u$ output channels. In this work, it is implemented as a two-layer pointwise multilayer perceptron.

Overall, the multiscale point cloud neural operator is formulated at the operator level, as in \cref{eq:NO_architecture}, and is discretized using traditional numerical discretizations: numerical quadrature for integration and local least squares reconstruction for gradient estimation. This allows the method to handle varying spatial resolutions, with discretization error controlled by local mesh size rather than by the learning procedure.

\paragraph{Computational cost} Let
$K:=(2p+1)^d$
denote the number of retained Fourier modes. The dominant floating-point operation
count for one forward evaluation is
\begin{equation}
\begin{split}
    C_{\rm total}  = \calO\Bigl(K L d_g N + L K d_g^2 + d L d_g^2 N\Bigr).
\end{split}
\end{equation}
The first term accounts for the forward and inverse
Fourier sums over the $N$ points, the second accounts for the mode-wise
channel mixing by the Fourier weight matrices, and the third accounts for
the local short-range operations, including gradient reconstruction and
pointwise feature transformations.
A detailed derivation is provided in the supplementary material. For fixed spatial
dimension $d$, Fourier truncation parameter $p$, latent width $d_g$, and
network depth $L$, the inference cost is linear in the number of points
$N$. The prefactor may nevertheless be substantial because the number of
retained Fourier modes, $(2p+1)^d$, grows rapidly with both $p$ and $d$.

\subsection{Training Objective}
\label{ssec:mpcno-training}
We train the M-PCNO in a supervised setting using a dataset  
\begin{equation}
    \left\{
        (f_i,\calD_i,u_i)
    \right\}_{i=1}^{n},
    \qquad
    u_i=\calG^\dagger(f_i,\calD_i),
\end{equation}
where the input-geometry pairs $(f_i,\calD_i)$ are sampled from a
distribution $\mu$. The population risk is the expected relative
$L^2$ error
\begin{equation}
\label{eq:pop_risk}
    \calJ(\theta)
    :=
    \bE_{(f,\calD)\sim\mu}
    \left[
        \frac{
            \left\|
                \calG^\dagger(f,\calD)
                -
                \calG_\theta(f,\calD)
            \right\|_{L^2(\calD)}
        }{
            \left\|
                \calG^\dagger(f,\calD)
            \right\|_{L^2(\calD)}
        }
    \right].
\end{equation}
In practice,  we minimize the corresponding empirical risk
\begin{equation}
\label{eq:emp_risk}
    \calJ_n(\theta)
    :=
    \frac{1}{n}
    \sum_{i=1}^{n}
    \frac{
        \left\|
            u_i-\calG_\theta(f_i,\calD_i)
        \right\|_{L^2(\calD_i)}
    }{
        \left\|u_i\right\|_{L^2(\calD_i)}
    }.
\end{equation}
When the fields are represented on point clouds, we approximate the
$L^2$ norms using equal point weights. This approximation is consistent with the surface $L^2$ norm when the points are approximately uniformly distributed with respect to the relevant surface measure; otherwise, it defines an equal-weight empirical $L^2$ norm.

\section{Numerical Study}
\label{sec:numerics}
In this section, we present numerical studies of the M-PCNO, with a particular focus on geometric generalization. Specifically,
\begin{enumerate}
    \item We first investigate the learning of integral operators with translation-invariant kernels defined on variable 2D curves, focusing on the approximation error.
    \item We then consider two exterior Laplace problems: a 2D
    Neumann-to-Dirichlet map and 3D potential flow. These
    operators involve both layer potentials and the solution of
    geometry-dependent boundary integral equations.
    \item Finally, we examine turbulent flow over 3D vehicles  to
    assess the performance of M-PCNO for a nonlinear PDE-induced operator
    on complex geometries.
\end{enumerate}
Detailed experimental setups are provided in the supplementary material. 

\subsection{Kernel Integral Problem}
\label{ssec:kernel_integral}
We first study integral operators defined on variable 2D curves:
$$\calG^{\dagger}: (f, \partial \Omega) \rightarrow u, \qquad u(x) = \int_{\partial \Omega} \kappa(x - y; n_x, n_y) f(y) \dd S_y, $$
where $\partial \Omega \subset \bR^2$ is a variable boundary geometry.  
We consider five representative kernels: the 2D Laplace single layer, double layer, modified double layer, and adjoint double layer kernels, and the 2D Stokeslet. Their definitions are given
in \cref{tab:kernel-short-range-approximation}.

The input functions $f$ are sampled from a Gaussian random field. We generate two classes of boundary geometries. The \textit{single-curve
dataset} consists of geometries with one closed curve, whereas the \textit{two-curve dataset} consists of two disjoint closed curves placed side
by side; see the left and right pairs of columns, respectively, in \cref{fig:kernel_integral}.
The second dataset therefore introduces an out-of-distribution change in both geometry and topology, from one connected component to two. The models are trained on the \textit{single-curve dataset} and are evaluated on both the single-curve in-distribution test set and the two-curve out-of-distribution test set.

\begin{figure}
    \centering
\includegraphics[width=0.9\linewidth]{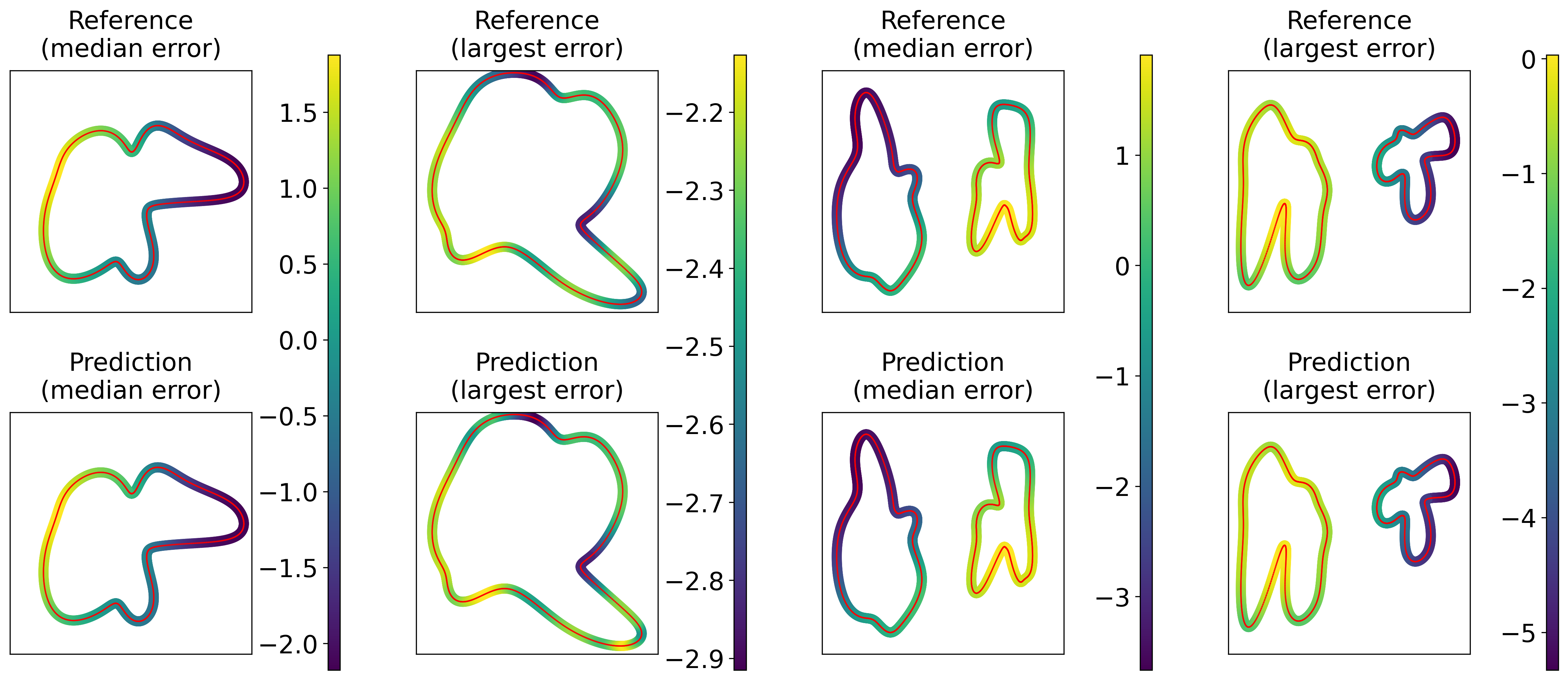}~~~~
    \caption{Representative results for learning the Laplacian single layer potential using a 5-layer M-PCNO with $p=32$ and $n=8000$. Each column shows the reference solution (top) and the prediction (bottom) for test cases with median and largest relative errors from the single-curve (left two) and two-curve (right two) test datasets.}
    \label{fig:kernel_integral}
\end{figure}

\begin{figure}
    \centering
\includegraphics[width=\linewidth]{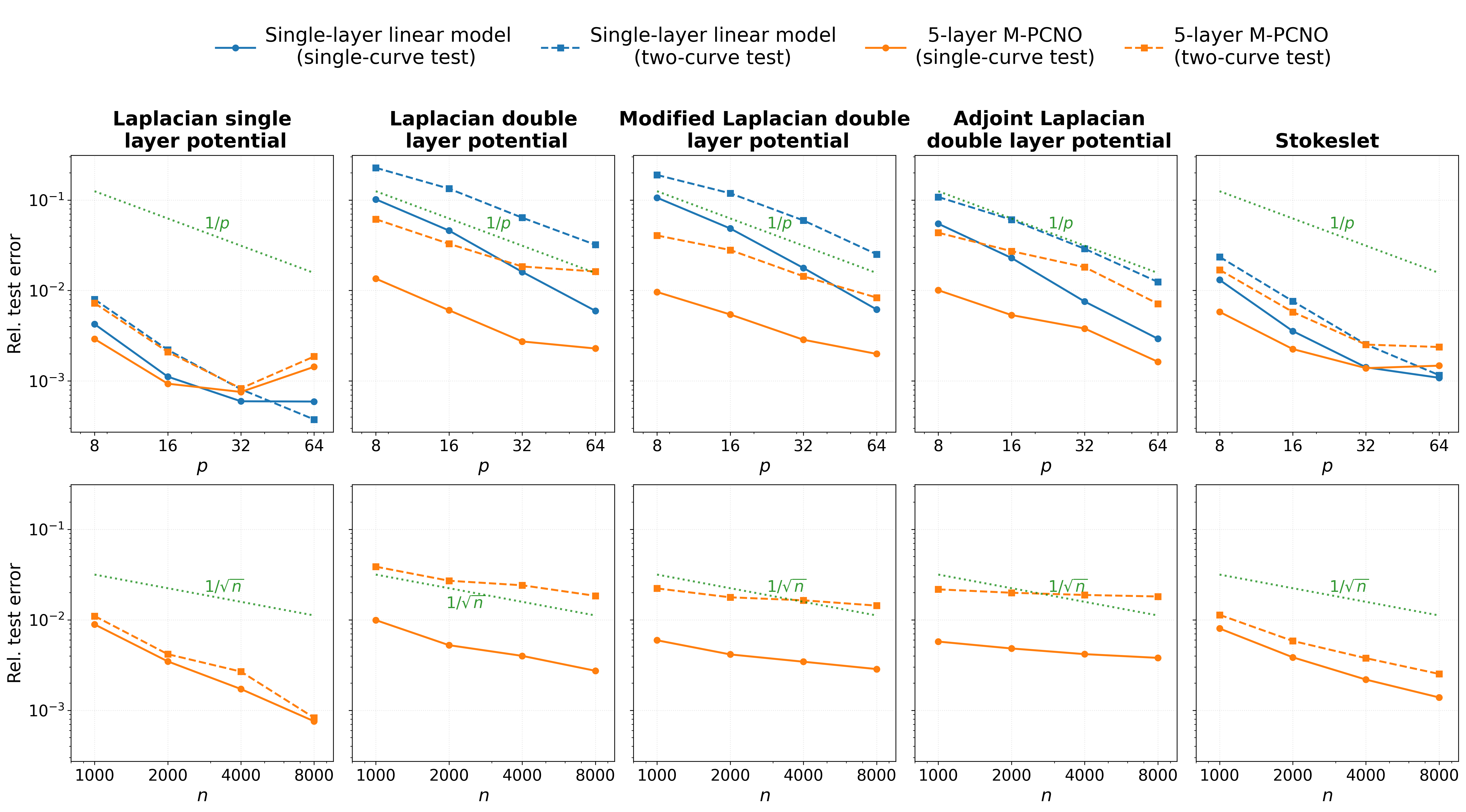}~~~~
    \caption{Kernel integrals: relative test errors as functions of the truncated mode number $p$ with training dataset size $n=8000$~(top row), and as functions of the training dataset size $n$ with $p=32$ fixed~(bottom row), for different models and test datasets. Each column corresponds to one kernel integral.}
    \label{fig:kernel_integral_error_plot}
\end{figure}

We compare two architectures. The first is a single multiscale point cloud neural layer with the nonlinear activation removed and serves as a baseline. The second is a 5-layer M-PCNO and is used to investigate the effects of depth and nonlinear composition. We vary the truncated mode number $p \in\{8, 16, 32, 64\}$ and training dataset size $n \in\{ 1000, 2000, 4000, 8000\}$.
The complete results are  summarized in \cref{fig:kernel_integral_error_plot}. Overall, M-PCNO generalizes
effectively across geometric and topological variations.  We observe that
\begin{enumerate}
\item For the single-layer linear model, the error scales as $p^{-1}$ in both single-curve in-distribution and two-curve out-of-distribution tests (see \cref{fig:kernel_integral_error_plot}-top), consistent with the theoretical result in \cref{theorem:approximation-error}.  
\item For 5-layer M-PCNO, the error is generally lower than that of the linear model and still roughly scales as $p^{-1}$. However, when 
$p$ becomes large, the available training data are insufficient to reliably train the more expressive model. As a result, the test error no longer decreases with increasing $p$ and instead plateaus (see \cref{fig:kernel_integral_error_plot}-top).
\item Increasing the training data size $n$ for 5-layer M-PCNO reduces the error following the scaling law $n^{-\beta}$. For the Laplacian single and double layer potentials and the Stokeslet, the convergence rate reaches the Monte Carlo rate ($\beta=\frac{1}{2}$), while other kernels exhibit slower convergence. We conjecture that kernels with weaker singularities converge more rapidly (see \cref{fig:kernel_integral_error_plot}-bottom).
\end{enumerate}

\subsection{Exterior Laplace Equation}
\label{sec:potential_flow}
We next consider two geometry-dependent exterior Laplace problems. Both problems involve kernel integrals and their resolvents, i.e., the solution of boundary integral equations, which are generally not translation-invariant.

\subsubsection{Exterior Neumann Problem}
Let $\Omega\subset\bR^2$ be a bounded domain with smooth boundary. We
consider
\begin{equation}
\begin{split}
       \Delta \Phi &= 0 \qquad \textrm{in} \qquad \Omega^c,\\
        \frac{\partial \Phi(x)}{\partial n_x} &= f(x) \qquad \textrm{on} \qquad \partial \Omega, \qquad 
        \int_{\partial \Omega} f(x) \dd x = 0, \\
        \lim_{|x|\to \infty }\Phi(x) &= 0.
\end{split} 
\end{equation}
where $n_x$ denotes the outward unit normal to $\Omega$. The goal is to learn the Neumann-to-Dirichlet map:
$$\calG^{\dagger}: (f, \partial \Omega) \rightarrow \Phi(x)\bigl|_{\partial\Omega}.$$ 
The solution admits the single layer potential representation
\begin{equation}
\label{eq:single_layer}
\Phi(x) = \int_{\partial \Omega} \eta(y)  \kappa(x - y)  \dd S_y, \qquad \kappa(x- y) = -\frac{1}{2\pi} \ln \lVert x - y \rVert_2,
\end{equation}
where $\eta$ is an unknown density on $\partial\Omega$. The density  
$\eta$ satisfies the following Fredholm integral equation of the second kind on  $\partial\Omega$:
\begin{equation}
\label{eq:Fredholm_integral_equation}
f(x) =    -  \frac{1}{2}\eta(x) +  \int_{\partial \Omega} \eta(y) \nabla_x \kappa(x - y) \cdot n_x \dd S_y. 
\end{equation}
Thus, the Neumann-to-Dirichlet map involves the composition of the operators related to \cref{eq:single_layer,eq:Fredholm_integral_equation}. Since the resolvent associated with \cref{eq:Fredholm_integral_equation} is geometry dependent and generally cannot be represented by a single
translation-invariant convolution.  Instead, it admits a Neumann series expansion involving repeated operator compositions, so sufficient depth is required for accurate approximation.

The experimental setup, including the geometries and input functions, is identical to that in \cref{ssec:kernel_integral}. The models are trained on the single-curve dataset and evaluated on both the single-curve and the two-curve dataset.

\begin{table}[htbp]
\begin{center}
\resizebox{0.85\textwidth}{!}{%
		\begin{tabular}{c|c|c|c|c}
			\Xhline{1.1pt}
			\diagbox{$L$}{$n$}  & $1000$ & $2000$ & $4000$ & $8000$\\ \Xhline{1.1pt}
            4 & 3.7753 \quad    9.3654   & 2.3512    \quad 8.6415   & 1.3535    \quad 6.111   & 0.9740    \quad 5.2895  \\
            \hline
			5 & 3.5967 \quad 9.3321   & 2.1442 \quad 7.3943   & 1.2108 \quad 5.764   & 0.9217 \quad 4.9039   \\ \Xhline{1.1pt}
		\end{tabular}
        }
	\end{center}
    \caption{Neumann-to-Dirichlet map for the exterior Laplacian learned with $p=32$ fixed for different layer number $L$ and training dataset size $n$. Each entry reports single-curve and two-curve Relative $L^2$ test errors ($\times 10^{-2}$).}
    \label{tab:exterior_neumann_layers}
\end{table}

To examine the effect of network depth, we fix $p=32$ and compare M-PCNOs with $L=4$ and $L=5$ layers over $n=\{1000,2000,4000,8000\}$ training samples. The results are reported in
\cref{tab:exterior_neumann_layers}. 
We next fix $L=5$ and vary $p \in\{ 8, 16, 32, 64 \}$ and $n \in \{ 1000, 2000, 4000, 8000\}$ to study the dependence of the test error on
spectral resolution and training set size.
The complete results are 
summarized in \cref{fig:exterior_neumann}.
We observe that
\begin{enumerate}
    \item
    At fixed $p=32$, increasing the depth from $L=4$ to $L=5$
    consistently reduces both the single-curve and two-curve test errors for every training-set size.
    This result is consistent with the need to represent a solution of boundary integral operator through repeated compositions.

    \item
    On the two-curve test set, the error initially decreases with $p$ at an
    empirical rate slightly slower than $p^{-1}$. At the largest value of
    $p$, the improvement plateaus or reverses, indicating that finite-data
    and optimization errors dominate the remaining spectral truncation
    error.
    The single-curve errors are substantially smaller and exhibit
    weaker dependence on $p$; therefore, low in-distribution error alone
    does not guarantee geometric or topological generalization.

    \item
    Over the sample sizes considered, the single-curve and two-curve test
errors decrease with $n$ and, for sufficiently large $p$, exhibit an
approximate power-law dependence on $n$, with different empirical slopes
for the two test distributions.
\end{enumerate}

\begin{figure}
    \centering
\includegraphics[width=0.49\linewidth]{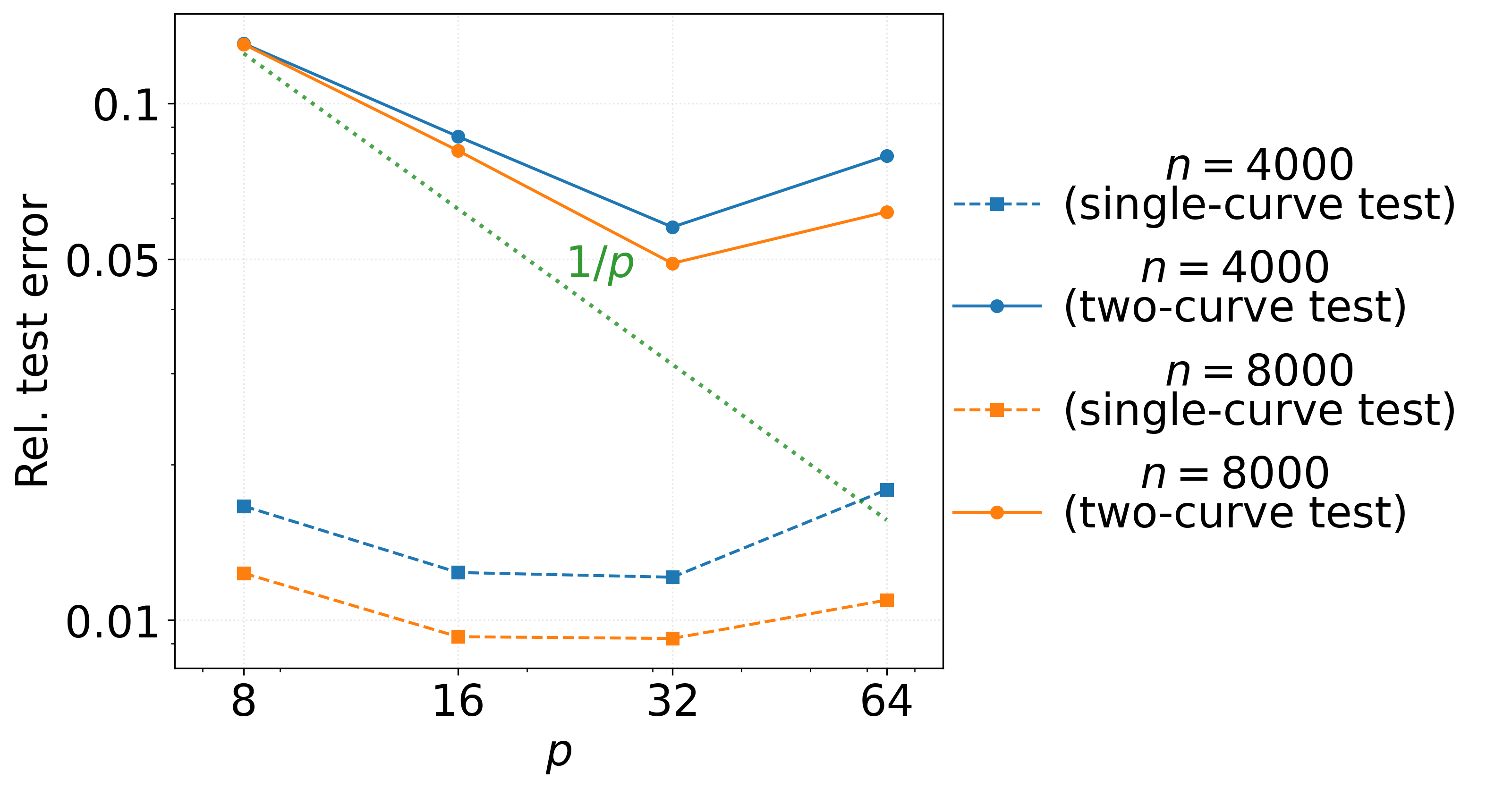}
\includegraphics[width=0.49\linewidth]{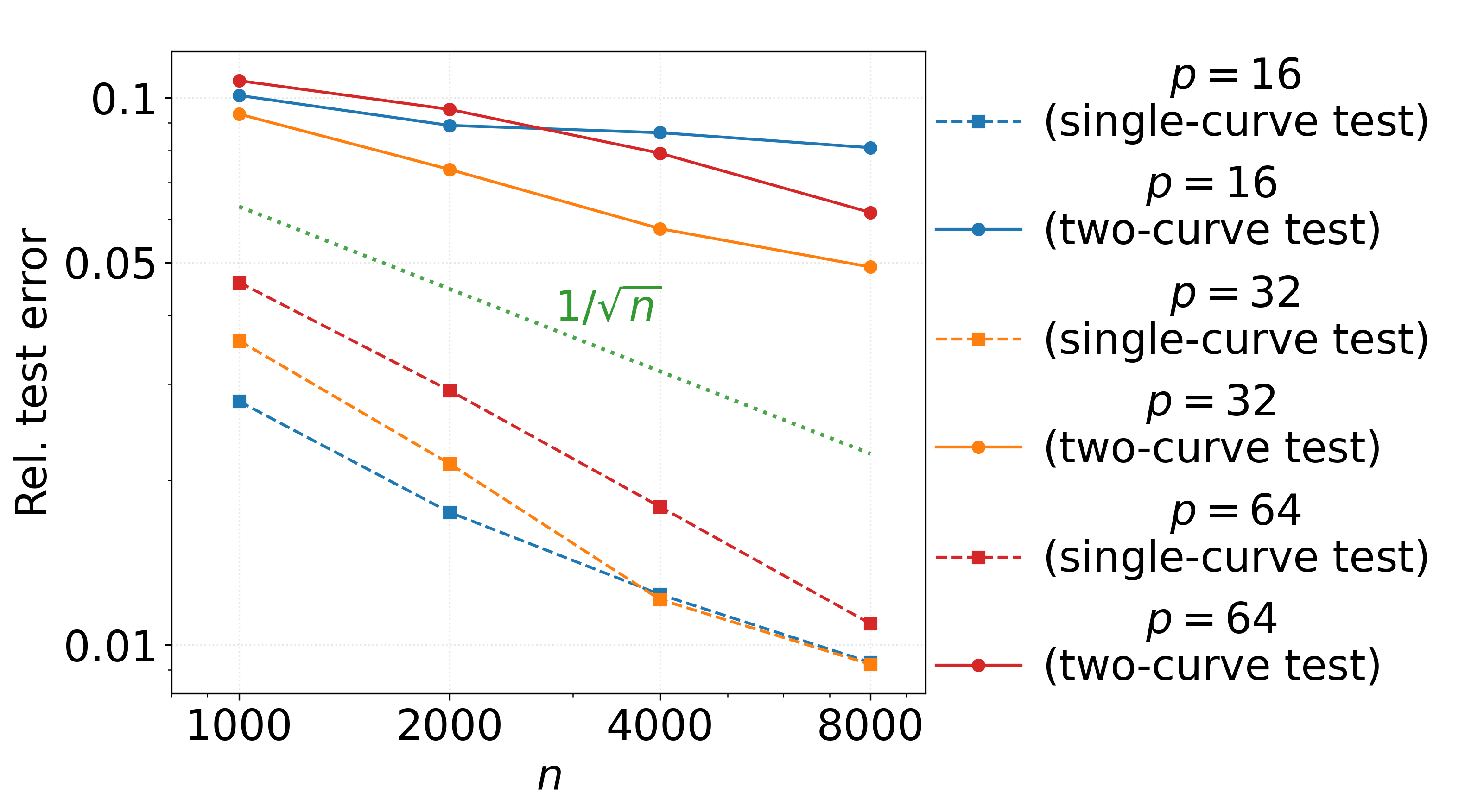}
    \caption{Neumann-to-Dirichlet map for the exterior Laplacian learned with a 5-layer M-PCNO: relative $L^2$ test errors as functions of the truncated mode number $p$ (left) and the training dataset size $n$ (right).}
    \label{fig:exterior_neumann}
\end{figure}

\subsubsection{Potential Flow Problem}
We next consider potential flow around a 3D object, discussed in the introduction \cref{eq:panel_phi_decomposition,eq:potential-flow-bie,eq:potential-flow-velocity,eq:pressure-coefficient}.  Our goal is to learn the aerodynamic map, from the inflow condition and boundary geometry to the pressure coefficient
$$\calG^{\dagger}: (v_{\infty}, \partial \Omega) \rightarrow C_p\Bigl|_{\partial\Omega}.$$
The freestream inflow condition is $v_{\infty} = [1,\,0,\,0]$ and scaling the inflow velocity does not affect $C_p$.

\begin{figure}
    \centering
\includegraphics[width=0.9\linewidth]{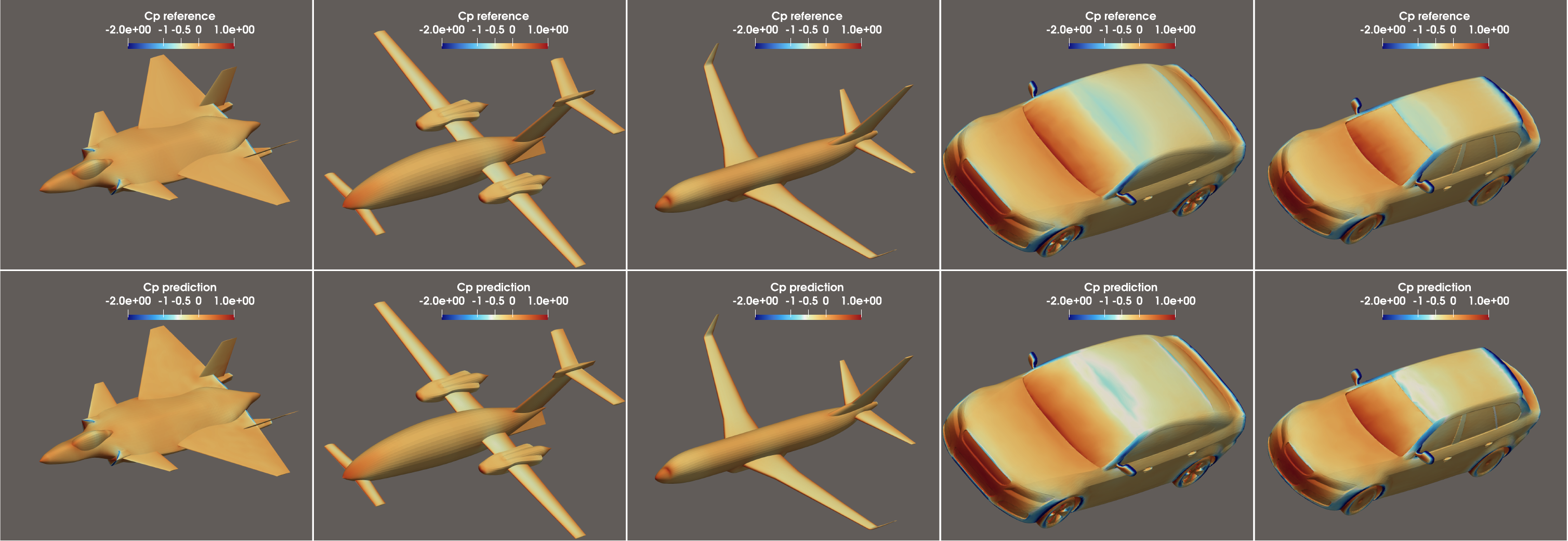}~~~~
    \caption{Representative results for the potential flow problem using a 5-layer M-PCNO with $p=16$ and $n=4000$. Each column shows the reference solution (top) and the prediction (bottom) for test cases with the largest relative error (leftmost, fighter jet), the median relative error (left-middle, turboprop), and three randomly selected samples.}
    \label{fig:potential_flow}
\end{figure}

The dataset contains approximately $5000$ ground-vehicle geometries from
DrivAerNet++~\cite{elrefaie2025drivaernetlargescalemultimodalcar} and $6000$
aircraft geometries (see \cref{fig:potential_flow}). Each geometry is discretized using
approximately $4\times10^4$ surface elements. The reference pressure
coefficient is computed using a panel method.

We train 5-layer M-PCNO models with $p = 16$ using $n \in\{ 1000, 2000, 4000\}$ training geometries randomly sampled from the full dataset. Performance is evaluated on a separate test set of $500$ geometries. The relative $L^2$ test errors reported in \cref{tab:potential_flow_N_vs_e} follow a scaling law proportional to $n^{-0.167}$. Representative results for the model trained with $n=4000$ are shown in \cref{fig:potential_flow}. The mean relative $L^2$ test error is $8.054\%$ across this diverse geometric dataset. 
These results demonstrate that a single surrogate model has the potential to accommodate a wide range of geometries.


\begin{table}[htbp]
    \begin{center}
        \begin{tabular}{c|ccc}
        \Xhline{1.1pt}
            $n$  & $1000$ & $2000$ & $4000$\\ \hline
            Rel. test error~($\times 10^{-2}$)  & 10.15   &   8.557 & 8.054   \\\Xhline{1.1pt}
        \end{tabular}
    \end{center}
    \caption{Potential flow problem with a 5-layer M-PCNO. Relative $L^2$ test errors ($\times 10^{-2}$) for different training dataset sizes $n$ with the truncated mode number fixed at $p=16$.}
    \label{tab:potential_flow_N_vs_e}
\end{table}

\subsection{Turbulent Flow Problem}

We finally consider steady turbulent flow around a 3D
vehicle. The time-averaged flow is modeled by the incompressible
Reynolds-averaged Navier-Stokes equations
\begin{align}
    (v\cdot\nabla)v+\frac{1}{\rho}\nabla p-\nu\nabla^2 v+\nabla\cdot \tau^R=0, \\ 
    \nabla\cdot v=0.
\end{align}
Here $v$ and $p$ denote the time-averaged velocity and pressure,
respectively, $\rho$ is the constant density, $\nu$ is the kinematic
viscosity, and $\tau^R$ is the Reynolds-stress tensor. The turbulence
closure is supplied by a $k$--$\varepsilon$ model. We impose the no-slip
condition $v=0$ on the vehicle and a uniform freestream velocity of
$72\,\mathrm{km/h}$ at the far-field boundary. The goal is to learn a surrogate map from the vehicle surface to the surface pressure:

\begin{equation}
 \calG^\dagger : \partial \Omega \mapsto p|_{\partial\Omega}.
\end{equation}

We use the \textit{ShapeNet-Car} dataset from \cite{umetani2018learning}, which contains sports cars, sedans, and SUVs collected from \cite{chang2015shapenet}.
After removing side mirrors,
spoilers, and tires, the dataset contains $611$ geometries, each represented by a surface mesh with approximately $3.7\times10^3$ points.

\begin{table}[htbp]
\centering
\begin{tabular}{c|cccc}
\toprule
Method &  GINO~\cite{li2024geometry} & Transolver~\cite{wu2024transolver} & MSPT~\cite{curvo2026mspt} & \textbf{M-PCNO} \\
\hline
Rel. test error~($\times 10^{-2}$) 
 & 7.12 
 & 7.45  
 & 7.41  
 & \textbf{6.36} \\
\bottomrule
\end{tabular}
\caption{Comparison on the \textit{ShapeNet-Car} surface pressure prediction task. The reported metric is the relative $L^2$ error. Results for baseline methods are taken from their original papers.}
\label{tab:shapenetcar_pressure}
\end{table}

We compare a 7-layer M-PCNO with the Fourier neural operator variant  GINO~\cite{li2024geometry}, and transformer-based architectures Transolver~\cite{wu2024transolver} and MSPT~\cite{curvo2026mspt}. The relative $L^2$ test errors are reported in~\cref{tab:shapenetcar_pressure}. The empirical distribution of test errors is illustrated in the left panel of \cref{fig:car-error}. 
To examine the large-error
outliers, the middle and right columns of \cref{fig:car-error} show the
test geometries with the largest and median relative errors, respectively.
We observe that 
\begin{enumerate}
    \item  Although its architecture is motivated by linear kernel integral
    operators, M-PCNO achieves a relative $L^2$ test error of $6.36\%$ on
    this nonlinear problem. This is competitive with respect to the methods compared in
    \cref{tab:shapenetcar_pressure}.

\item The error distribution in the left panel of \cref{fig:car-error} contains several notable outliers. For the largest-error geometry, the principal discrepancy is concentrated near an unusually thin region at the bottom of the vehicle. This observation suggests that atypical geometric features may contribute to large prediction errors and motivates the development of geometry-aware uncertainty quantification and out-of-distribution detection.

\item 
Preprocessing and inference for one geometry require approximately $1.17$ and $0.054$ seconds on GPU, respectively. By comparison, the parallelized Navier-Stokes solver reported in \cite{umetani2018learning} requires approximately $50$ minutes per geometry. Although these timings were obtained under different hardware settings, they indicate a substantial reduction in online evaluation time.
\end{enumerate}

\begin{figure}[htbp]
     \centering
     \includegraphics[width=0.9\textwidth]{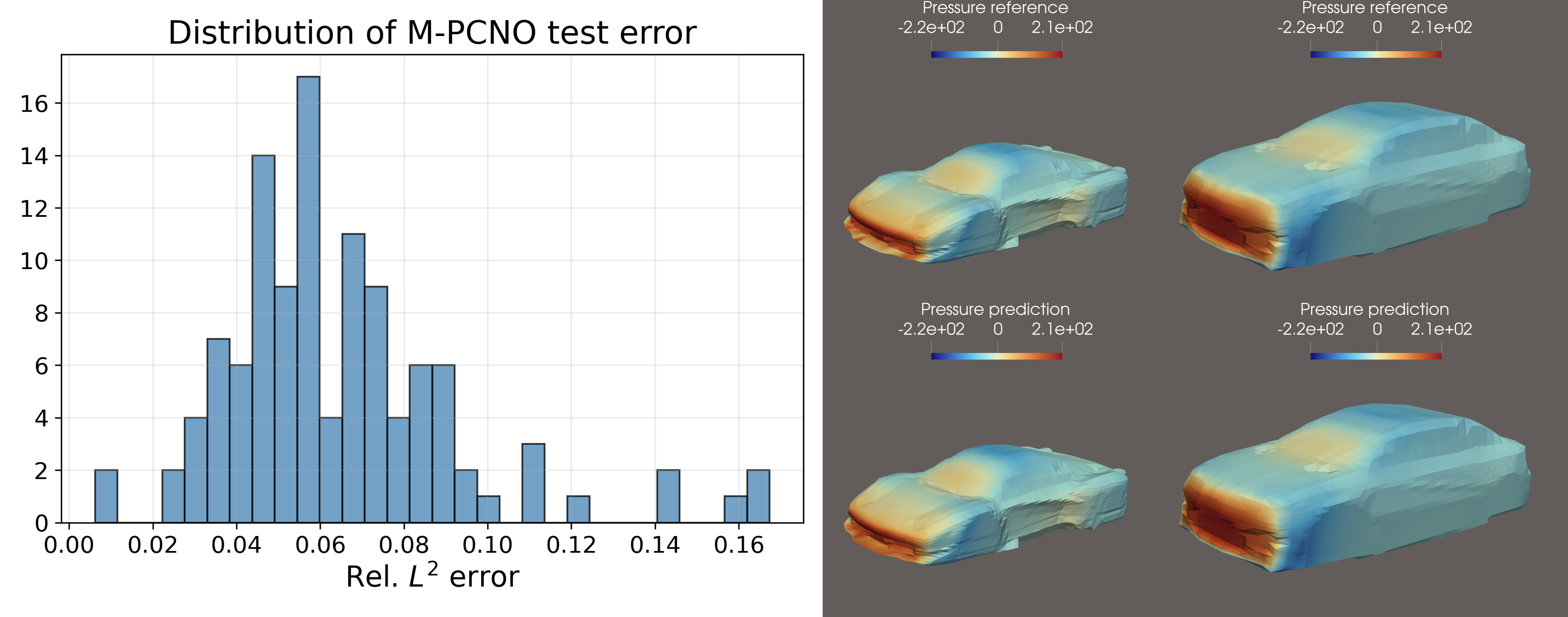}\\
     \caption{Error analysis for 3D turbulent flow over a vehicle. 
     From left to right: the distribution of test errors, the test case with the largest error, and a representative median error case. For the two visualized cases, the reference solution is shown in the top row and the M-PCNO prediction in the bottom row.}
     \label{fig:car-error}
\end{figure}

\section{Conclusion}
\label{sec:conclusion}

We study neural operators on variable, nonparametric geometries from a kernel integral perspective. This viewpoint motivates the M-PCNO architecture and yields approximation error bounds for the underlying linear integral operators. Numerical experiments demonstrate robust generalization across diverse geometries together with favorable computational efficiency. An important theoretical direction is to extend this framework to nonlinear PDE-induced operators by analyzing compositions of kernel-integral layers and nonlinear activations. On the application side, developing reliable uncertainty quantification and out-of-distribution detection for previously unseen geometries remains a key challenge.

 \section*{Acknowledgments}
We acknowledge the support of the high-performance computing platform of Peking University.

\appendix

\section{Proof of \cref{theorem:approximation-error}}
\label{sec:proof-theorem:approximation-error}
\begin{proof}
    Since $\kappa \in L^1(B_2)$ and is periodic, it has Fourier coefficients
    \begin{equation}
        |\hat{\kappa}_k| = \Bigl| \int_{B_2} \kappa(x) e^{-2\pi i  k \cdot x} \dd x \Bigr| \leq  \lVert \kappa \rVert_{L^1(B_2)}, \qquad k  \in \bZ^d.
    \end{equation}
Next, we show that $\kappa_{\rm long}$ is well defined on $\bR^d$ and bounded. 
Decomposing $\bR^d$ into translated cubes and using periodicity of $\kappa$, we obtain
{\small\begin{equation}
\label{eq:kappa_bound}
    \begin{split}
        \int_{\bR^d} |\kappa(x-y)|\rho_\delta(y)\dd y
&=\sum_{n\in\bZ^d}\int_{[-\frac{1}{2},\frac{1}{2}]^d+n}|\kappa(x-y)|\rho_\delta(y)\dd y\\
&\leq\lVert\kappa\rVert_{L^1(B_2)}\sum_{n\in\bZ^d}\max_{[-\frac{1}{2},\frac{1}{2}]^d+n}\rho_\delta(y)\\
&\leq \lVert \kappa \rVert_{L^1(B_2)} 
\frac{2^d}{(2\pi \delta^2)^ {d/2}} 
\sum_{n_1 =0}^{\infty}\cdots\sum_{n_{d} =0}^{\infty} \max_{ [-\frac{1}{2},\frac{1}{2}]^d + n }  e^{-\frac{\lVert y\rVert_2^2}{2\delta^2}}   \\
&\leq\lVert\kappa\rVert_{L^1(B_2)}\frac{2^d}{(2\pi \delta^2)^ {d/2}}  \Bigl(1+\sum_{n_1=1}^\infty e^{-\frac{1}{2\delta^2}(n_1 - \frac{1}{2})^2}\Bigr)^{d}\\
    &\leq C(\kappa,d,\delta).
    \end{split}
\end{equation}
}In the second inequality, we use symmetry to restrict $n$ in the first quadrant.  In the third inequality, we use the decomposition  $\lVert y\rVert_2^2=\sum_{i=1}^d y_i^2$ and bound each component separately.
So $\kappa_{\rm long} = \kappa * \rho_{\delta}$ is well-defined everywhere and is bounded. 
Next, we establish each property separately.

\paragraph{Proof of \textnormal{(P1)}}
For the mollified kernel $\kappa_{\rm long}$, we compute its Fourier coefficients:
\begin{equation}
\label{eq:kappa_k}
\begin{split}
        \widehat{\kappa_{\rm long}}_k &=  \int_{B_2}\kappa_{\rm long}e^{-2\pi i k \cdot x } \dd x= \int_{B_2}\dd x\int_{\mathbb{R}^d}\kappa(x-y)\rho_\delta(y)e^{-2\pi ik \cdot x} \dd y\\
&=\int_{\mathbb{R}^d}\rho_\delta(y)e^{-2\pi i k \cdot y}\int_{B_2-y} \kappa(x)e^{-2\pi i k \cdot x} \dd x \dd y\\
&=\hat{\kappa}_k e^{-2\pi^2\delta^2\lVert k\rVert_2^2}. 
\end{split}
\end{equation}
The exchange of integrals is justified by Tonelli's and Fubini's theorems, using the bound in \cref{eq:kappa_bound}. Hence, mollification damps the Fourier coefficients of $\kappa$ by a Gaussian factor.
We now derive an $L^{\infty}$ bound on the Fourier tail of $\kappa_{\rm long}$, i.e., on the partial sum over all modes $k\in \bZ^d:\lVert k\rVert_\infty>p$ :
{\small \begin{align}
\sum_{k:\lVert k\rVert_\infty>p} |\widehat{\kappa_{\rm long}}_k| &\leq  \lVert \kappa \rVert_{L^1(B_2)}   2^d d \sum_{k_1 =0}^{\infty}\cdots \sum_{k_{d-1} =0}^{\infty}\sum_{k_{d} =p+1}^{\infty} e^{-2\pi^2\delta^2\lVert k \rVert_2^2} \nonumber\\
&\leq  \lVert \kappa \rVert_{L^1(B_2)} 2^d d  \Bigl(1 + \int_{0}^{\infty} e^{-2\pi^2\delta^2k^2} dk   
\Bigr)^{d-1} \Bigl(\int_{p}^{\infty} e^{-2\pi^2\delta^2k^2} dk   
\Bigr) \label{eq:k_long_tail} \\
&\leq \lVert \kappa \rVert_{L^1(B_2)}  2^d d \Bigl(1 + \frac{1}{2\sqrt{2\pi} \delta}   
\Bigr)^{d-1} \frac{e^{-2\pi^2\delta^2p^2}}{4\pi^2\delta^2 p} \quad \textrm{using } \delta <\frac{1}{2}  \nonumber\\
&\leq\frac{d\lVert\kappa\rVert_{L^1(B_2)}}{2\pi^2}(1+\frac{1}{\sqrt{2\pi}})^{d-1}\frac{e^{-2\pi^2\delta^2p^2}}{\delta^{d+1} p}. \nonumber
\end{align}
}For the second inequality, we substitute \cref{eq:kappa_k} and consider $k$ in the first quadrant with $k_d > p$. 
For the third inequality, we use the decomposition $\lVert k\rVert_2^2=\sum_{i=1}^d k_i^2$ and bound each component separately. 
Finally, we consider the linear operator related to the smooth part
\begin{equation}
\label{eq:operator_long_tail}
\begin{aligned}
    \lVert\widehat{\calK}_{\rm long}f - \calK_{\rm long}f\rVert_{L^\infty(\partial\Omega)}&= 
\sup_{x\in\partial \Omega} \Bigl|\int_{\partial\Omega}\sum_{k:\lVert k\rVert_\infty>p}\widehat{\kappa_{\rm long}}_k e^{2\pi ik\cdot(x-y)}f(y)\,\dd y\Bigr| \\
&\leq \int_{\partial\Omega}\sum_{k:\lVert k\rVert_\infty>p}|\widehat{\kappa_{\rm long}}_k||f(y)|\,\dd y\\
&\leq \sum_{k:\lVert k\rVert_\infty>p} |\widehat{\kappa_{\rm long}}_k| \lVert f\rVert_{L^1(\partial\Omega)}. 
\end{aligned}
\end{equation}
Combining \cref{eq:k_long_tail,eq:operator_long_tail} proves
\cref{eq:kappa_long_conclusion}.

\paragraph{Proof of \textnormal{(P2)}}
For the short-range component, let $0 < \delta_p \leq \frac{1}{2}$. For every $x \in B_2$ such that $\lVert x\rVert_2 > \delta_p$, we decompose the integral into two regions:
{\small \begin{align}
    \kappa_{\rm short}(x) 
    &= \int_{\bR^d} \Bigl(\kappa(x) - \kappa(x-y) \Bigr)\rho_{\delta}(y) \dd y \nonumber\\
    &= \int_{\lVert y \rVert_2 < \delta_p} \Bigl(\kappa(x) - \kappa(x-y) \Bigr)\rho_{\delta}(y)\dd y + \int_{\lVert y \rVert_2 > \delta_p} \Bigl(\kappa(x) - \kappa(x-y) \Bigr)\rho_{\delta}(y) \dd y. \label{eq:short-range-split}
\end{align}
}For the first term in \cref{eq:short-range-split}, a Taylor expansion of $\kappa(x-y)$ gives, for some $\theta\in(0,1)$,
{\small
\begin{align}
\Bigl|\int_{\lVert y \rVert_2 < \delta_p} \Bigl(\kappa(x) &- \kappa(x-y) \Bigr)\rho_{\delta}(y)\dd y\Bigr|
\label{eq:kappa_short_part1}\\
&\leq \Bigl|\int_{\lVert y \rVert_2 < \delta_p} \bigl(\nabla\kappa(x) \cdot y\bigr)\rho_\delta(y) \dd y\Bigl| 
+ \frac{1}{2} \int_{\lVert y \rVert_2 < \delta_p} \bigl\lVert \nabla^2 \kappa\bigl(x-\theta(x,y)y\bigr)\bigr\rVert_2 \bigl\lVert y \bigr\rVert_2^2 \rho_{\delta}(y) \dd y\nonumber\\
&\leq 0 + \frac{d\delta^{2}}{2} \sup_{\lVert z\rVert_2>\lVert x\rVert_2-\delta_p}\lVert \nabla^2 \kappa(z)\rVert_2 \nonumber \quad \textrm{using assumption} ~\ref{assumption:kappa_assumption}\\
&\leq \frac{d\delta^{2}}{2} C_{\kappa}(\frac{1}{\lVert x\rVert_2-\delta_p} )^{d+1}. \nonumber
\end{align}
}In the second inequality, we use $\lVert x - \theta y\rVert_2 \geq \lVert x\rVert_2 - \delta_p$ and $\int_{\mathbb{R}^d} \lVert y\rVert_2^2\rho_\delta(y)\dd y=\mathrm{tr}[\int_{\mathbb{R}^d} yy^T\rho_\delta(y)\dd y]=d\delta^{2}$.
For the second term in \cref{eq:short-range-split}, we have the decomposition
{\small\begin{equation}
\label{eq:kappa_short_part2}
\begin{split}
\Bigl|\int_{\lVert y \rVert_2 > \delta_p} \Bigl(\kappa(x) &- \kappa(x-y) \Bigr)\rho_{\delta}(y)\dd y\Bigr| \\
&\leq \int_{\lVert y \rVert_2 > \delta_p} |\kappa(x)|\rho_{\delta}(y)\dd y + 
\int_{\lVert y \rVert_2 > \delta_p} |\kappa(x-y)|\rho_{\delta}(y)\dd y.
\end{split}
\end{equation}
}The first term in \cref{eq:kappa_short_part2} is a Gaussian tail. Using spherical coordinates, we have
{\small\begin{align}
        |\kappa(x)|\int_{\lVert y \rVert_2 > \delta_p} \rho_\delta(y)\dd y &= |\kappa(x)|\frac{1}{(2\pi)^{d/2}}\int_{\lVert z \rVert_2 > \delta_p/\delta} e^{-\frac{\lVert z \rVert_2^2}{2}} \dd z  \nonumber
        \\&=|\kappa(x)|\frac{2\pi^{d/2}}{(2\pi)^{d/2}\Gamma(d/2)}\int_{\delta_p/\delta}^{\infty}r^{d-1}e^{-\frac{r^2}{2}}dr \label{eq:kappa_short_part2_term1}\\
    & \leq |\kappa(x)|\frac{2^{1-d/2}}{\Gamma(d/2)}2^{d/2-1}\Gamma(\frac{d}{2},\frac{\delta_p^2}{2\delta^2}) \quad \textrm{using assumption}~ \ref{assumption:kappa_assumption} \nonumber\\
    & \leq \frac{C_{\kappa}}{\lVert x\rVert_2^{d-1}}\frac{1}{\Gamma(d/2)}\Gamma(\frac{d}{2},\frac{\delta_p^2}{2\delta^2}),   
    \nonumber 
    \end{align}
}where $\Gamma$ is the Gamma function. To bound the second term in \cref{eq:kappa_short_part2} (with $\delta_p \leq \frac{1}{2}$), we decompose
$\{y : \lVert y \rVert_2 > \delta_p\}$ into $\{\lVert y \rVert_2 > \delta_p\} \cap[-\frac{1}{2},\frac{1}{2}]^d$ and the reminder, covered by the translated cubes $\bigcup_{n\in \bZ^d\backslash\{0\}} \{[-\frac{1}{2},\frac{1}{2}]^d + n\}$, and then use the periodicity of $\kappa$:
{\small  
\begin{align}
     \int_{\lVert y \rVert_2 > \delta_p} |\kappa(x-y)| \rho_{\delta}(y) \dd y 
     &= \sum_{n\in \bZ^d} \int_{\{\lVert y \rVert_2 > \delta_p\} \cap \{[-\frac{1}{2},\frac{1}{2}]^d + n\}} |\kappa(x-y)| \rho_{\delta}(y) \dd y  \nonumber\\
&\leq \lVert \kappa \rVert_{L^1(B_2)} \sum_{n\in \bZ^d} \max_{\{\lVert y \rVert_2 > \delta_p \}\cap \{[-\frac{1}{2},\frac{1}{2}]^d + n\}}  \rho_{\delta}(y)  \label{eq:kappa_short_part2_term2_1}
\\
&= \lVert \kappa \rVert_{L^1(B_2)} 
\Bigl(
\frac{1}{(2\pi \delta^2)^ {d/2}} e^{-\frac{ \delta_p^2 }{2 \delta^2}} + 
\sum_{n\in \bZ^d \setminus\{0\}} \max_{ [-\frac{1}{2},\frac{1}{2}]^d + n }  \rho_{\delta}(y)  
\Bigr). \nonumber
\end{align}
}For the sum over translated cubes in \cref{eq:kappa_short_part2_term2_1}, we have
{\small  
\begin{equation}
\label{eq:kappa_short_part2_term2_2}
    \begin{split}
\sum_{n\in \bZ^d \setminus\{0\}} \max_{ [-\frac{1}{2},\frac{1}{2}]^d + n }  \rho_{\delta}(y)  
&\leq  \frac{2^d d}{(2\pi \delta^2)^ {d/2}} 
\sum_{n_1 =0}^{\infty}\cdots \sum_{n_{d-1} =0}^{\infty}\sum_{n_{d} =1}^{\infty} \max_{ [-\frac{1}{2},\frac{1}{2}]^d + n }  e^{-\frac{\lVert y\rVert_2^2}{2\delta^2}}  \\
&\leq \frac{2^d d}{(2\pi \delta^2)^ {d/2}}  \bigl(1+\sum_{n_1=1}^\infty e^{-\frac{1}{2\delta^2}(n_1 - \frac{1}{2})^2}\bigr)^{d-1}\sum_{n_d=1}^\infty e^{-\frac{1}{2\delta^2}(n_d - \frac{1}{2})^2}
\\
&\leq \frac{2^d d}{(2\pi \delta^2)^ {d/2}} \frac{e^{-\frac{1}{8\delta^2}}} {(1-e^{-\frac{1}{8\delta^2}})^{d}}. 
    \end{split}
\end{equation}
}Here the first inequality follows by symmetry (restricting $n$ to the first quadrant with $n_d \geq 1$); the second inequality uses the decomposition $\lVert y\rVert_2^2=\sum_{i=1}^d y_i^2$ and bounds each component separately; and the last uses $\sum_{n_i=1}^\infty e^{-\frac{1}{2\delta^2}(\frac{2n_i-1}{2})^2}\leq\sum_{n_i=1}^\infty e^{-\frac{n_i}{8\delta^2}}=\frac{e^{-\frac{1}{8\delta^2}}}{ 1-e^{-\frac{1}{8\delta^2}} }$, which follows from $(2n_i-1)^2\geq n_i,\forall n_i\geq1$.
Substituting \cref{eq:kappa_short_part2_term2_2} into \cref{eq:kappa_short_part2_term2_1} and using $\delta_p \leq \frac{1}{2}$, we can bound the second term in \cref{eq:kappa_short_part2} as
{\small  
\begin{equation}
\label{eq:kappa_short_part2_term2}
    \begin{split}
     \int_{\lVert y \rVert_2 > \delta_p} |\kappa(x-y)| \rho_{\delta}(y) \dd y 
&\leq  \frac{\lVert \kappa \rVert_{L^1(B_2)} }{(2\pi \delta^2)^ {d/2}}
\Bigl(
 1 + \frac{2^d d}{(1-e^{-\frac{1}{8\delta^2}})^{d}}
\Bigr)   e^{-\frac{ \delta_p^2 }{2 \delta^2}}. 
    \end{split}
\end{equation}}

Substituting \cref{eq:kappa_short_part1,eq:kappa_short_part2_term1,eq:kappa_short_part2_term2} into \cref{eq:short-range-split} and choosing 
\begin{equation}
    \delta_p = c_d\lVert x\rVert_2 \leq \min\{\lVert x\rVert_2,\frac{1}{2}\} \qquad c_d=\frac{1}{\sqrt d}
\end{equation}
 yields the following bound for the short-range component, valid for any $x\in B_2\setminus\{0\}$:
{\small\begin{align*}
    \Bigl|\kappa_{\rm short}(x)\Bigl|\leq \frac{dC_{\kappa}\delta^2}{2(1-c_d)^{d+1}}\frac{1}{\lVert x\rVert_2^{d+1}}+
\frac{\lVert \kappa \rVert_{L^1(B_2)}}{(2\pi)^ {d/2}}  \bigl(1+ \frac{2^d d}{(1 - e^{-\frac{1}{8\delta^2}})^d}\bigr)\frac{e^{- \frac{c_d^2 \lVert x\rVert_2^2}{2 \delta^2}}}{\delta^d}+\frac{C_{\kappa}}{\lVert x\rVert_2^{d-1}\Gamma(d/2)}\Gamma\Bigl(\frac{d}{2},\frac{c_d^2}{2\delta^2}\lVert x\rVert_2^2\Bigl).
\end{align*}
}
Bringing the upper bound for the Gamma function from \cref{lemma:Gamma_function} into the last term, using $c_d=d^{-1/2}<1$ we obtain that for any $\lVert x \rVert_2 \geq \delta$ and $\delta<\frac{1}{2}$,  
{\small
\begin{equation}
\label{eq:kappa_short}
    \begin{aligned}
        \Bigl|\kappa_{\rm short}(x)\Bigl|
    &\leq\frac{dC_{\kappa}}{2(1-c_d)^{d+1}}\frac{\delta^2}{\lVert x\rVert_2^{d+1}}+\Bigl(\frac{\lVert\kappa\rVert_{L^1(B_2)}}{(2\pi)^{d/2}}(1+\frac{2^dd}{(1-e^{-\frac{1}{8\delta^2}})^d})+\frac{C_{\kappa}C_{d/2}}{\Gamma(\frac{d}{2})}\bigl(\frac{c_d^{d-2}}{2^{d/2-1}}\frac{\delta^2}{\lVert x\rVert_2}+\frac{2d\delta^{d+2}}{\lVert x\rVert_2^{d+1}}\bigr)\Bigr)\frac{e^{-\frac{1}{2d}\frac{\lVert x\rVert^2}{\delta^2}}}{\delta^d}\\
    &\leq\frac{dC_{\kappa}}{2(1-c_d)^{d+1}}\frac{\delta^2}{\lVert x\rVert_2^{d+1}}+\Bigl(\frac{\lVert\kappa\rVert_{L^1(B_2)}}{(2\pi)^{d/2}}(1+\frac{2^dd}{(1-e^{-1/2})^d})+\frac{C_{\kappa}C_{d/2}}{\Gamma(\frac{d}{2})}
    \bigl(
    \frac{1}{2^{d/2}}+d
    \bigr)\Bigl)\frac{e^{-\frac{1}{2d}\frac{\lVert x\rVert^2}{\delta^2}}}{\delta^d}.
    \end{aligned}
\end{equation}
}
\paragraph{Proof of \textnormal{(P3)}}
Using the estimation about spatial decay of $\kappa_{\rm short}$ in \cref{eq:kappa_short} with $\epsilon>\delta$ and the boundary tail estimate in \cref{lemma:boundary-tail} lead to 
{\small
\begin{align}
    \lVert \calK_{\rm short} f - \calK_{\rm short}^\epsilon f\rVert_{L^\infty(\partial\Omega)}&=\sup_x \Bigl|\int_{\partial\Omega\backslash B_\epsilon(x)}\kappa_{\rm short}(x-y)f(y)\,\dd S_y\Bigr|
    \nonumber\\
    &\leq \int_{\partial\Omega\backslash B_\epsilon(x)} \frac{1}{\lVert x - y\rVert_2^{d+1}} dS_y \lVert f\rVert_{L^{\infty}(\partial\Omega)}\frac{dC_{\kappa}\delta^2}{2(1-c_d)^{d+1}} + c_{\rm short}\frac{e^{- \frac{\epsilon^2}{2d \delta^2}}}{\delta^{d}} \lVert f\rVert_{L^1(\partial\Omega)}  \label{eq:local}
    \\
&\leq  \frac{C_{\rm tail}}{\epsilon^{2}} \frac{dC_{\kappa}\delta^2}{2(1-c_d)^{d+1}}\lVert f\rVert_{L^\infty(\partial\Omega)} + c_{\rm short}\frac{e^{- \frac{\epsilon^2}{2d \delta^2}}}{\delta^{d}}  \lVert f\rVert_{L^1(\partial\Omega)}, \nonumber
\end{align}
}
where $C_{\rm tail}$ depends on $C_L$ and $d$. We then approximate $\calK$ by
\begin{equation}
    \widehat{\calK}= \widehat{\calK}_{\rm long}+\widehat{\calK}^{\epsilon}_{\rm short}.
\end{equation}
Applying the error bounds in \cref{eq:operator_long_tail,eq:local,eq:local-short-range-error}, we decompose the $L^\infty$ error into three contributions:
{\small
\begin{equation}
\label{eq:operator_diff_0}
    \begin{aligned}
       \lVert \calK f-\widehat{\calK} f\rVert_{L^\infty(\partial\Omega)}\leq& \lVert \calK_{\rm long} f-\widehat{\calK}_{\rm long} f\rVert_{L^\infty(\partial\Omega)}+\lVert \calK_{\rm short} f-{\calK}_{\rm short}^\epsilon f\rVert_{L^\infty(\partial\Omega)}+\lVert {\calK}_{\rm short}^\epsilon f-\widehat{\calK}^{\epsilon}_{\rm short} f\rVert_{L^\infty(\partial\Omega)}\\
    \leq& \frac{C_{\rm long}}{\delta^{d+1}p}e^{-2\pi^2\delta^2p^2}\lVert f\rVert_{L^1(\partial\Omega)}+ C_{\rm tail}\frac{dC_{\kappa}\delta^2}{2\epsilon^{2}(1-c_d)^{d+1}}\lVert f\rVert_{L^\infty(\partial\Omega)} + c_{\rm short}\frac{e^{- \frac{c_d^2 \epsilon^2}{2 \delta^2}}}{\delta^{d}}\lVert f\rVert_{L^1(\partial\Omega)}
    \\
    &+C_{\rm short} \epsilon^{d+q} \lVert f\rVert_{W^{2,\infty}(\partial\Omega)}
    \\
    \leq& C\Bigl(\frac{e^{-2\pi^2\delta^2p^2}}{\delta^{d+1}p}   +  \frac{ \delta^2}{ \epsilon^{2}} +  \frac{e^{- \frac{c_d^2 \epsilon^2}{2 \delta^2}}}{\delta^{d}} + \epsilon^{d+q}\Bigr), 
    \end{aligned}
\end{equation}}
\noindent 
where $C = \max\Bigl\{
C_{\rm long}\lVert f\rVert_{L^1(\partial\Omega)}, 
\frac{C_{\rm tail}dC_{\kappa}}{2(1-c_d)^{d+1}}\lVert f\rVert_{L^\infty(\partial\Omega)}, 
c_{\rm short}\lVert f\rVert_{L^1(\partial\Omega)}
,C_{\rm short}  \lVert f\rVert_{W^{2,\infty}(\partial\Omega)}\Bigr\}$. By \cref{lemma:boundary-tail}, the measure of $\partial \Omega$ is bounded in terms of $C_L$ and $d$. Consequently, 
$C$ depends on $C_\kappa$, $\lVert \kappa \rVert_{L^1(B_2)}$, $C_L$,  $d$, and $\lVert f\rVert_{W^{2,\infty}(\partial\Omega)}$.

We then choose $\delta$ and $\epsilon$ to balance the two polynomially decaying terms in \cref{eq:operator_diff_0} while ensuring that the exponential terms decay rapidly. Specifically, set 
 \begin{equation}
  \delta = p^{-\gamma},\qquad  \epsilon = \delta^t = p^{-\gamma t},  
  \qquad t = \frac{2}{q+d+2},
  \end{equation} 
  with parameters $\gamma \in (0,1)$, and $p > 2^{\frac{1}{\gamma}}$.
These choices ensure $\delta < \frac{1}{2}$ and $\epsilon \geq \delta$ (since $t\in(0,1)$), as required for the short-range approximation \cref{eq:kappa_short_conclusion}. 
Substituting into \cref{eq:operator_diff_0} yields
{\small\begin{equation*}
    \begin{split}
       \lVert \calK f-\widehat{\calK} f\rVert_{L^\infty(\partial\Omega)} &\leq C\Bigl(\frac{e^{-2\pi^2\delta^2p^2}}{\delta^{d+1}p} + \delta^{-d}e^{-\frac{1}{2}c_d^2\delta^{2t-2}} +2\delta^{t(q+d)}\Bigl)\\
&\leq C \Bigl(p^{-1+\gamma(d+1)}e^{-2\pi^2p^{2-2\gamma}} + p^{d\gamma}e^{-\frac{1}{2}c_d^2p^{\gamma(2-2t)}} +2p^{-\gamma t(q+d)}\Bigl).
    \end{split}
\end{equation*}}
\noindent The bound consists of two exponentially decaying terms and one polynomially decaying term. For sufficiently large $p$, the polynomial term
$p^{-\gamma \Bigl(1 + \frac{q + d - 2}{q+d+2}\Bigr)}$ dominates the overall error.
\end{proof}

\begin{lemma}
\label{lemma:Gamma_function}
    For any $\alpha>0$, there exists a constant $C_\alpha$ depending only on $\alpha$ such that for all $x>0$,
$\Gamma(\alpha,x)=\int_x^\infty t^{\alpha-1}e^{-t} dt\leq C_\alpha (x^{\alpha-1} + x^{-1})e^{-x}$.
\end{lemma}
    \begin{proof}
        We prove the claim by induction. For every $n\in\bZ_+$, we show that \begin{equation*}
            \forall\alpha\in(n-1,n],\quad \exists C_{\alpha} \quad \textrm{such that}\quad \Gamma(\alpha,x)\leq C_\alpha (x^{\alpha-1} + x^{-1})e^{-x}\quad\forall x>0.
        \end{equation*}
        For the base case $n=1$, we have $\alpha\in(0,1]$, it follows that $t^{\alpha-1}\leq x^{\alpha-1}$ for all $t\geq x>0$. Therefore, $\Gamma(\alpha,x)=\int_x^\infty t^{\alpha-1}e^{-t} dt\leq x^{\alpha-1} \int_x^\infty e^{-t}dt = x^{\alpha-1}e^{-x}$.
  For the induction step, assume the claim holds for $n=k-1\in\bZ_+$. For $n=k\in\bZ_+$, integration by parts gives
    \begin{align*}
\Gamma(\alpha,x)=\int_x^\infty t^{\alpha-1}e^{-t} dt&=-t^{\alpha-1}e^{-t}\Bigl|_x^\infty +(\alpha-1)\int_x^\infty t^{\alpha-2} e^{-t}dt\\
        &\leq x^{\alpha-1}e^{-x}+(\alpha-1)C_{\alpha-1}(x^{\alpha-2} + x^{-1})e^{-x}\\
        &\leq (1+2(\alpha-1)C_{\alpha-1})(x^{\alpha-1} + x^{-1})e^{-x}. 
    \end{align*}
    Here the last inequality uses $x^{\alpha-2} + x^{-1}\leq 2(x^{\alpha-1} + x^{-1})$ under the condition $\alpha > 1$, since $x^{\alpha-2} \leq x^{\alpha-1}$ when $x\geq 1$ and $x^{\alpha-2} \leq x^{-1}$ when $x<1$.
    By induction on $n$, the result holds for all $\alpha > 0$.
\end{proof}

\begin{lemma}[Boundary tail estimate]
\label{lemma:boundary-tail}
Let $\partial\Omega \subset B:=\left[0,\frac12\right]^d$
be a compact embedded $(d-1)$-dimensional Lipschitz hypersurface with
Lipschitz character bounded by $C_L$. Then there exists a constant
$C_{\rm geo}=C_{\rm geo}(d,C_L)>0$
such that 
\begin{equation}
\label{eq:surface-growth}
    \left|
        \partial\Omega\cap B_\epsilon(x)
    \right|
    \leq
    C_{\rm geo}\epsilon^{d-1},
    \qquad
    x\in\partial\Omega,
    \quad
    0<\epsilon <{\rm diam}(\partial \Omega),
\end{equation}
where $\lvert\cdot\rvert$ denotes the $(d-1)$-dimensional surface measure.
Then $\lvert\partial\Omega\rvert
    \leq
    C_{\rm geo}
    \left(\frac{\sqrt d}{2}\right)^{d-1}$.
Moreover, for every $\epsilon>0$,
\begin{equation}
\label{eq:boundary-tail-estimate}
    \sup_{x\in\partial\Omega}
    \int_{\partial\Omega\setminus B_\epsilon(x)}
    \frac{1}{\lVert x-y\rVert_2^{d+1}}
    \,\dd S_y
    \leq
    C_{\rm tail}\epsilon^{-2},
\qquad \mathrm{with} \qquad 
    C_{\rm tail}
    :=
    \frac{2^{d+1}}{3}C_{\rm geo}.
\end{equation}
\end{lemma}

\begin{proof}

The surface-growth estimate in \cref{eq:surface-growth} follows from the
upper Ahlfors regularity of compact Lipschitz submanifolds; see
\cite[Proposition~1.12]{rataj2019curvature}. Since
$\operatorname{diam}(B)=\frac{\sqrt d}{2}$, we have $\partial\Omega \subset B_{\sqrt d/2}(x)$
for every $x\in\partial\Omega$. Applying \cref{eq:surface-growth} with
$r=\sqrt d/2$ proves the upper bound of $\lvert\partial\Omega\rvert$.
 To prove \cref{eq:boundary-tail-estimate}, we first decompose
$\partial\Omega\setminus B_\epsilon(x)$ into dyadic annuli
\[
    A_j(x)
    :=
    \left\{
        y\in\partial\Omega:
        2^j\epsilon
        \leq
        \lVert x-y\rVert_2
        <
        2^{j+1}\epsilon
    \right\}.
\]
Using \cref{eq:surface-growth}, we obtain
\begin{align*}
    \int_{\partial\Omega\setminus B_\epsilon(x)}
    \frac{\dd S_y}{\lVert x-y\rVert_2^{d+1}}
    \leq
    \sum_{j\geq0}
    \frac{|A_j(x)|}{(2^j\epsilon)^{d+1}}\leq
    C_{\rm geo}
    \sum_{j\geq0}
    \frac{(2^{j+1}\epsilon)^{d-1}}
    {(2^j\epsilon)^{d+1}}
    \leq
    \frac{2^{d+1}}{3}
    C_{\rm geo}\epsilon^{-2}.
\end{align*}
\end{proof}

\section{Kernel Analysis}
In this section, we analyze the kernels used throughout this work (see, e.g., \cref{tab:kernel-short-range-approximation}). 
First, in \cref{appendix:kernel_properties}, we verify that these kernels satisfy the regularity assumptions required for \cref{theorem:approximation-error}.
Next, in \cref{appendix:short-range}, we derive local short-range asymptotic approximations of these kernel integrals. 
The analysis shows that the leading-order contribution takes the form
\begin{equation}
\int_{\partial\Omega\cap B_\epsilon(x)} \kappa'(x-y) f(y) \dd S_y = M_0(x) f(x) +  M_1(x) \cdot \nabla_{\calD} f(x)  + \mathcal{O}(\epsilon^{d}),
\end{equation}
where the remainder is higher order in 
$\epsilon$ (at least one order beyond the intrinsic dimension of $\partial\Omega$). The geometric moments $M_0$ and $M_1$ are polynomial functions of the outward normal $n_x$ and the curvature $\mathrm{tr}[\nabla_{\calD} n_x]$. 
Finally, in \cref{appendix:panel}, we present a panel method evaluation of the kernel integrals used in \cref{ssec:kernel_integral}.

\subsection{Kernel Properties Verification}
\label{appendix:kernel_properties}
In this subsection, we verify that the kernels in \cref{tab:kernel-short-range-approximation} satisfy the regularity assumptions in \Cref{theorem:approximation-error}. Specifically, for a bounded box
$B_2=\prod_{i=1}^d[-l_i,l_i],$
we require $\kappa\in L^1(B_2)\cap C^2(B_2\setminus\{0\})$ and the pointwise bounds
\begin{equation}\label{eq:kappa_assumption_recall}
\|\nabla^k\kappa(x)\|_2 \le \frac{C}{\|x\|_2^{k+d-1}}, \qquad k=0,1,2,\quad x\in B_2\setminus\{0\},
\end{equation}
for some constant $C>0$. All kernels are smooth away from the origin.  After rescaling the box if necessary (absorbing the scaling into the constant $C$), it suffices to verify \eqref{eq:kappa_assumption_recall} for $0<\|x\|_2\le 1$. Let $r = \lVert x \rVert_2$. For $0\leq r \leq 1$, we use the elementary inequalities $|\log r|\le r^{-1}$ and $r^{-m} \leq r^{-(m+1)}$.
\begin{itemize}
\item For the 2D Laplacian single layer kernel
$\kappa(x)=-(2\pi)^{-1}\log r$, we have $$\nabla\kappa(x)=-(2\pi)^{-1}\frac{x}{r^2}, \quad
\nabla^2\kappa(x)=-(2\pi)^{-1}\Bigl(\frac{1}{r^2}I_2-\frac{2}{r^4}xx^T\Bigr).$$ 
Hence, for $0<r\le 1$, 
$|\kappa(x)|\leq r^{-1}$, $\|\nabla\kappa(x)\|_2\lesssim r^{-1}\le r^{-2}$,$\|\nabla^2\kappa(x)\|_2\lesssim r^{-2}\le r^{-3}$.

\item For the 2D modified Laplacian double layer kernel $\kappa(x)=\frac{x}{2\pi r^2}$, we have 
\begin{align*}
    \nabla\kappa(x) = \frac{1}{2\pi}(\frac{1}{r^2}I_2-\frac{2}{r^4}xx^T), \quad
    \nabla^2\kappa(x)[h] = \frac{1}{2\pi}\bigl( \frac{8x^Th}{r^6}xx^T - \frac{2(xh^T+hx^T+x^ThI_2)}{r^4} \bigl).
\end{align*}
Hence, for $0<r\le 1$, we have $\|\kappa(x)\|_2\lesssim r^{-1}$, $\|\nabla\kappa(x)\|_2\lesssim r^{-2}$, and $\|\nabla^2\kappa(x)\|_2\lesssim r^{-3}$.
The Laplacian double layer and adjoint double layer kernels are obtained by multiplying by bounded unit normals (e.g., $n_x$ and $n_y$), so they have the same singular order and satisfy the same bounds.

\item For the 2D Stokeslet, 
$
\kappa(x)=\frac{1}{4\pi}\Bigl(-\log r\,I_2+\frac{xx^T}{r^2}\Bigr),
$
we have \begin{align*}
    \nabla\kappa(x)[h] =&\frac{1}{4\pi}\bigl( -\frac{x^Th}{r^2}I_2 + \frac{1}{r^2}(hx^T+xh^T)-\frac{2x^Th}{r^4}xx^T\bigl),\\
    \nabla^2\kappa(x)[h,k]=&\frac{1}{4\pi}\bigl(-\frac{k^Th}{r^2}I_2+\frac{2(x^Th)(x^Tk)}{r^4}I_2+\frac{1}{r^2}(hk^T+kh^T)\\&-\frac{2x^Tk}{r^4}(hx^T+xh^T)-\frac{2x^Th}{r^4}(kx^T+xk^T)-\frac{2k^Th}{r^4}xx^T+\frac{8(x^Th)(x^Tk)}{r^6}xx^T\bigl).
\end{align*}
Thus, for $0<r\leq1$, we have $\lVert\kappa(x)\rVert_2\leq r^{-1}$,  $\|\nabla\kappa(x)\|_2\lesssim r^{-1}\leq r^{-2}$ and $\|\nabla^2\kappa(x)\|_2\lesssim r^{-2} \leq r^{-3}$. 
\item For the 3D Laplacian single layer kernel, $\kappa(x)=(4\pi)^{-1}/r$, we have 
\begin{align*}
    \nabla\kappa(x)=-(4\pi)^{-1}x/r^3, \quad \nabla^2\kappa(x)=-(4\pi)^{-1}\Bigl(\frac{1}{r^3}I_3-\frac{3}{r^5}xx^T\Bigr).
\end{align*}
Hence, for $0<r\le 1$, we have 
$|\kappa(x)|\lesssim r^{-1}\le r^{-2}$, $
\|\nabla\kappa(x)\|_2\lesssim r^{-2}\le r^{-3}$, and $
\|\nabla^2\kappa(x)\|_2\lesssim r^{-3}\le r^{-4}$.

\item For the 3D modified Laplacian double layer kernel, $\kappa(x)=(4\pi)^{-1}x/r^3$, we have 
\begin{align*}
    & \nabla\kappa(x) = \frac{1}{4\pi}(\frac{1}{r^3}I_3-\frac{3}{r^5}xx^T), \\
    & \nabla^2\kappa(x)[h] = \frac{1}{4\pi}\bigl( \frac{15x^Th}{r^7}xx^T - \frac{3}{r^5}(xh^T+hx^T+x^ThI_3) \bigl).
\end{align*}
Thus, for $0<r\le 1$, $\|\kappa(x)\|_2\lesssim r^{-2}$, $ \lVert \nabla\kappa(x)\rVert_2 \lesssim r^{-3}$, and $\|\nabla^2\kappa(x)\|_2\lesssim r^{-4}$. As above, the Laplacian double layer and adjoint double layer kernels obtained by multiplying by bounded unit normals have the same singular order and satisfy the same bounds.
\item For the 3D Stokeslet,
$
\kappa(x)=\frac{1}{8\pi}\Bigl(\frac{1}{r}I_3+\frac{xx^T}{r^3}\Bigr),
$ we have 
\begin{align*}
    \nabla\kappa(x)[h] =& \frac{1}{8\pi}\Bigl( -\frac{x^Th}{r^3}I_3 + \frac{1}{r^3}(hx^T+xh^T)-\frac{3x^Th}{r^5}xx^T\Bigl),\\
    \nabla^2\kappa(x)[h,k]=&\frac{1}{8\pi}\Bigl(-\frac{k^Th}{r^3}I_3+\frac{3(x^Th)(x^Tk)}{r^5}I_3+\frac{1}{r^3}(hk^T+kh^T)\\&-\frac{3x^Tk}{r^5}(hx^T+xh^T)-\frac{3x^Th}{r^5}(kx^T+xk^T)-\frac{3k^Th}{r^5}xx^T+\frac{15(x^Th)(x^Tk)}{r^7}xx^T\Bigl).\\
\end{align*}
Therefore, for $0<r\le 1$,  $\|\kappa(x)\|_2\lesssim r^{-1}\le r^{-2}$, $\|\nabla\kappa(x)\|_2\lesssim r^{-2}\le r^{-3}$ and $\|\nabla^2\kappa(x)\|_2\lesssim r^{-3}\le r^{-4}$.
\end{itemize}

\subsection{Short-Range Asymptotic Approximation}
\label{appendix:short-range}
We now derive short-range asymptotic approximations for the kernel integrals
\begin{equation}
\int_{\partial\Omega\cap B_{\epsilon}(x)} \kappa'(x-y, n_x, n _y) f(y) \dd S_y.
\end{equation}
Fix a target point $x$. We introduce a local coordinate system by mapping $y$ to 
\begin{equation}
\label{eq:local_reference_coord}
    \tilde{y} = Q(y - x),
\end{equation}
so that $x$ is mapped to the origin and the tangent and normal directions at $x$ form the coordinate axes. 

We first consider the two-dimensional case. Let $\tau_x$ and $n_x$ denote the tangent and outward unit normal at $x$ (right-hand convention). Define the orthogonal transformation matrix by
$ Q^T = [\tau_x \,\, n_x] $.
In the local coordinate, the boundary segment $\partial\Omega\cap B_\epsilon(x)$ can be parameterized as
$\tilde{y}(s) = \bigl(s, h(s)\bigr), \, s\in[-\epsilon, \epsilon]$, where  $h(0)=0$ and $h'(0)=0$. The outward unit normal and the surface measure are given by  $\tilde{n}_{\tilde{y}} = \frac{(h'(s), -1)}{\sqrt{1 + h'(s)^2}}$ and $\dd\tilde{y} = \sqrt{1 + h'(s)^2} \dd s$. Applying Taylor expansions about $s=0$, we have 
\begin{equation}
\label{eq:expansion-2d}
    \begin{split}
        h(s) &= \frac{h''(0)}{2} s^2 +  \frac{h^{(3)}(0)}{6} s^3 + \calO(s^4), \\
\frac{1}{s^2 + h(s)^2} &= \frac{1}{s^2 (1 + \frac{h''(0)^2}{4}s^2 + \mathcal{O}(s^3))} = \frac{1}{s^2} -  \frac{h''(0)^2}{4} + \mathcal{O}(s), \\
\sqrt{1 + h'(s)^2} &= 1 + \frac{h''(0)^2}{2}s^2 + \mathcal{O}(s^3),  \\
f(y) &= \tilde{f}\bigl((s, h(s))\bigr) =   f(x) + \nabla {f}(x) \cdot \tau_x s  + \calO(s^2). 
    \end{split}
\end{equation}
The tangential (curve) gradient of the outward normal satisfies
\begin{equation}
\label{eq:nabla_D_n_y_2d}
    \begin{split}
\nabla_{\calD} n_y \bigl|_{y=x} &= Q^T\nabla_{\tilde{\calD}} \tilde{n}_{\tilde{y}}\bigl|_{\tilde{y}=\tilde{x}}Q
= Q^T\begin{bmatrix}
    h^{''}(0)&0\\
    0&0
\end{bmatrix}Q. 
    \end{split}
\end{equation}
Here, $\nabla_{\cal D}$ denotes the tangential (curve) gradient. Locally at $x$, it reduces to $\nabla_{\cal D}\bigl|_x=(I - n_x\otimes n_x) \nabla = \tau_x \partial_s$. From \cref{eq:nabla_D_n_y_2d},  we therefore identify
$h''(0) = \mathrm{tr}\bigl[\nabla_{\calD} n_y\bigl|_{y=x}\bigr]$.
Using these expansions in \cref{eq:expansion-2d}, we derive the leading-order contributions of the short-range integrals for each kernel, summarized below.

\begin{itemize}
    \item For the 2D Laplace single layer potential, we have
\begin{align*}
 \int_{\partial \Omega \cap B_{\epsilon}(x)} \frac{-1}{2\pi}\log{\lVert x-y\rVert_2} f(y) \dd S_y 
&= \frac{-1}{4\pi} \int_{-\epsilon}^{\epsilon} \log{(s^2+h(s)^2)} \tilde{f}((s, h(s)))\sqrt{1+h'(s)^2} \dd s \\
&=  -\frac{1}{\pi}f(x)(\epsilon\log\epsilon-\epsilon) + o(\epsilon^2).
\end{align*}

 \item For the 2D Laplace double layer potential, we have 
\begin{align*}
 \int_{\partial \Omega \cap B_{\epsilon}(x)} \frac{(x-y)\cdot n_y}{2\pi\lVert x - y \rVert^2_2} f(y) \dd S_y 
&=  \int_{-\epsilon}^{\epsilon} \frac{(-s, -h(s))\cdot(h'(s), -1)}{2\pi (s^2 + h(s)^2) } \tilde{f}((s, h(s))) \dd s \\
&=  -\frac{h''(0)}{2\pi}f(x)\epsilon + \calO(\epsilon^3).
\end{align*}

\item For the 2D  modified Laplace double layer potential, we have 
 
\begin{equation*}
\begin{split}
     \int_{\partial\Omega \cap B_{\epsilon}(x)} \frac{x - y}{2\pi \|x - y\|_2^2} f(y) \, \dd S_y &= Q^T \int_{-\epsilon}^{\epsilon} \frac{(-s, -h(s))}{2\pi (s^2 + h(s)^2) } \sqrt{1 + h'(s)^2} \tilde{f}((s, h(s))) \dd s \\
 &=  -\frac{\epsilon}{2\pi}Q^T\begin{bmatrix}
2\nabla f(x) \cdot \tau_x \\
h''(0)f(x)
\end{bmatrix} + \mathcal{O}(\epsilon^2)\\
&=  -\frac{\epsilon}{2\pi} 
 \Bigl(2\nabla_{\cal D} f(x) + 
 h''(0)f(x)  n_x\Bigr)
+ \mathcal{O}(\epsilon^3).
\end{split}
\end{equation*}

\item For the 2D  adjoint Laplace double layer potential, applying the same local approximation as for the modified Laplace double layer potential yields

\begin{align*}
\int_{\partial \Omega \cap B_{\epsilon}(x)} \frac{(y-x)\cdot n_x}{2\pi\lVert x - y \rVert^2_2} f(y) \dd S_y 
&=  -\frac{h''(0)}{2\pi}f(x)\epsilon +  \calO(\epsilon^3).
\end{align*}

\item For the 2D  Stokeslet, we have

\begin{align*}
   & \int_{\partial\Omega\cap B_\epsilon(x)}\frac{1}{4\pi}\Bigl(-\log{\lVert x-y\rVert_2}I_2+\frac{(x-y)(x-y)^T}{\lVert x-y\rVert_2^2}\Bigl)f(y)\dd s \\
   =& \frac{-1}{2\pi}(\epsilon\log\epsilon-\epsilon)f(x) + o(\epsilon^2) + \int_{-\epsilon}^\epsilon Q^T\frac{1}{4\pi}\Bigl(\frac{1}{s^2+h(s)^2}\begin{bmatrix}
       s^2 & sh(s) \\ sh(s) & h(s)^2
   \end{bmatrix}\Bigl)Q\tilde{f}((s, h(s)))\sqrt{1+h'(s)^2}\dd s \\
   =&\frac{-1}{2\pi}(\epsilon\log\epsilon-\epsilon)f(x) + \frac{1}{2\pi}Q^T\begin{bmatrix}
       1 & 0 \\ 0 & 0
   \end{bmatrix}Qf(x)\epsilon+o(\epsilon^2)\\
   =& \frac{-1}{2\pi}(\epsilon\log\epsilon-\epsilon)f(x) + \frac{1}{2\pi}\bigl(I_2 - n_x  n_x^T\bigr)f(x)\epsilon+o(\epsilon^2).
\end{align*}
\end{itemize}

We then consider the three-dimensional case. 
Let $\tau_{x,1}$, $\tau_{x,2}$, and $n_x$ denote two orthonormal tangent directions and the outward unit normal at $x$, respectively. Define the orthogonal transformation matrix $Q$ in \cref{eq:local_reference_coord} by 
$Q^T = [\tau_{x,1}\, \tau_{x,2}\, n_x].$
In the local coordinate, the local surface patch $\partial\Omega\cap B_{\epsilon}(x)$
can be parameterized as  
$\tilde{y}(s) = (s, h(s)), s = (s_1\,, s_2) \in B_{\epsilon}(0)$ where $h((0,0))=0$ and $\nabla h((0,0))=0$. The outward unit normal is $\tilde{n}_{\tilde{y}} = \frac{(\partial_1h(s), \partial_2h(s), -1)}{\sqrt{1 + \partial_1h(s)^2 + \partial_2h(s)^2}}$ and the surface measure satisfies $\dd\tilde{y} = \sqrt{1 + \partial_1h(s)^2 + \partial_2h(s)^2} \dd s$. Applying Taylor expansions about $s = (0,0)$, we have 
\begin{equation}
\label{eq:expansion-3D}
    \begin{split}
      h(s) &= \frac{1}{2} s^T \nabla^2 h(0) s + \calO(s^3), \\
(s^T s + h(s)^2)^{-3/2} &= \frac{1}{(s^Ts)^{3/2}} \Bigl(1 - \frac{3}{8}\frac{(s^T\nabla^2h(0) s)^2}{s^T s}\Bigr) +   \mathcal{O}(1), \\
(s^Ts+h(s)^2)^{-1/2}&=\frac{1}{(s^Ts)^{1/2}}\Bigl(1-\frac{1}{8}\frac{(s^T\nabla^2h(0) s)^2}{s^T s}\Bigl)+\mathcal{O}(\lVert s\rVert_2^2),\\
\sqrt{1 + \partial_1h(s)^2 + \partial_2h(s)^2} &= 1 + \frac{1}{2}s^T \nabla^2 h(0) \nabla^2 h(0)s + \mathcal{O}(\lVert s\rVert_2^3),  \\
f(y) &= \tilde{f}\bigl((s, h(s))\bigr) =   f(x) + \sum_{i=1}^2\nabla {f}(x) \cdot \tau_{x,i} s_i  + \calO(\lVert s\rVert_2^2).   
    \end{split}
\end{equation}
The tangential (surface) gradient of the outward normal satisfies
\begin{equation}
\label{eq:nabla_D_n_y}
    \begin{split}
\nabla_{\calD} n_y \bigl|_{y=x} &= Q^T\nabla_{\tilde{\calD}} \tilde{n}_{\tilde{y}}\bigl|_{\tilde{y}=\tilde{x}}Q
= Q^T\begin{bmatrix}
    \partial_{11}h(0)&\partial_{12}h(0)&0\\
    \partial_{21}h(0)&\partial_{22}h(0)&0\\
    0&0&0
\end{bmatrix}Q. 
    \end{split}
\end{equation}
Here $\nabla_{\cal D}$ denotes the tangential (surface) gradient. Locally, $\nabla_{\cal D}\bigl|_x =(I - n_x\otimes n_x) \nabla = \sum_{i=1}^2\tau_{x,i} \partial_{s_i}$. Consequently, 
$\Delta h(0) = \mathrm{tr}\bigl[\nabla_{\calD} n_y\bigl|_{y=x}\bigr]$.
Using these expansions in \cref{eq:expansion-3D}, we now derive the leading-order contributions of the short-range integrals for each kernel.

\begin{itemize}
    \item For the 3D Laplace single layer potential, we have 
\begin{align*}
 \int_{\partial \Omega \cap B_{\epsilon}(x)} \frac{1}{4\pi \lVert x-y\rVert_2} f(y) \dd S_y 
&= \frac{1}{4\pi}\int_{B_{\epsilon}(0)} (s^Ts+h(s)^2)^{-1/2}\sqrt{1 + \partial_1 h(s)^2 + \partial_2 h(s)^2} \tilde{f}((s, h(s))) \dd s \\
&= \frac{1}{4\pi}\int_{B_\epsilon(0)}(\frac{1}{(s^Ts)^{1/2}}+\mathcal{O}(\lVert s\rVert_2))(f(x)+\sum_{i=1,2}\nabla f(x)\cdot\tau_{x,i}s_i+\mathcal{O}(\lVert s\rVert_2^2))\dd s \\
&=\frac{\epsilon}{2}f(x) + \mathcal{O}(\epsilon^3).
\end{align*}

 \item For the 3D Laplace double layer potential, we have 
\begin{align*}
 \int_{\partial \Omega \cap B_{\epsilon}(x)} \frac{(x-y)\cdot n_y}{4\pi\lVert x - y \rVert^3_2} f(y) \dd S_y 
&=  -\frac{1}{4\pi}\int_{B_\epsilon(0)}\frac{s^T\nabla h(s)-h(s)}{(s^Ts+h(s)^2)^{3/2}}(f(x)+\sum_{i=1}^2\nabla f(x)\cdot\tau_{x,i}s_i+\mathcal{O}(\lVert s\rVert_2^2))\dd s\\
&=-\frac{1}{4\pi}\int_{B_\epsilon(0)}\frac{1}{(s^Ts)^{3/2}}(1-\frac{3}{8}\frac{(s^T\nabla^2h(0)s)^2}{s^Ts}) \frac{1}{2}s^T\nabla^2h(0)s f(x)\dd s +\mathcal{O}(\epsilon^3) \\
& = -\frac{1}{4\pi}\int_{B_\epsilon(0)}\frac{s^T\nabla^2h(0)s}{2(s^Ts)^{3/2}} f(x) \dd s  +\mathcal{O}(\epsilon^3)\\
&=-\frac{\epsilon}{8}\Delta h(0)f(x)+\mathcal{O}(\epsilon^3).
\end{align*}

\item For the 3D modified Laplace double layer potential, we have 
\begin{equation*}
\begin{split}
     \int_{\partial\Omega \cap B_{\epsilon}(x)} \frac{x - y}{4\pi \|x - y\|_2^3} f(y) \, \dd S_y &= -Q^T \int_{B_\epsilon(0)} \frac{(s,h(s))}{4\pi (s^Ts + h(s)^2)^{3/2} } \sqrt{1 + \partial_1 h(s)^2 + \partial_2 h(s)^2} \tilde{f}((s, h(s))) \dd s \\
     &= -Q^T \int_{B_\epsilon(0)} \frac{(s,\frac{1}{2}s^T\nabla^2 h(0) s)}{4\pi (s^Ts)^{3/2}}\bigl(f(x) + \sum_{i=1}^2\nabla {f}(x) \cdot \tau_{x,i} s_i\bigr) \dd s  + \mathcal{O}(\epsilon^3)\\
 &=  -\frac{\epsilon}{8}Q^T\begin{bmatrix}
2\nabla {f}(x) \cdot \tau_{x,1} \\
2\nabla {f}(x) \cdot \tau_{x,2}\\
\Delta h(0)f(x)
\end{bmatrix} + \mathcal{O}(\epsilon^3)\\
&=  -\frac{\epsilon}{8} 
 \Bigl(2\nabla_{\calD} f(x)  + 
 \Delta h(0)f(x)  n_x\Bigr)
+ \mathcal{O}(\epsilon^3).
\end{split}
\end{equation*}
\item For the 3D adjoint Laplace double layer potential, applying the same local approximation as for the modified Laplace double layer potential yields
\begin{align*}
\int_{\partial \Omega \cap B_{\epsilon}(x)} \frac{(y-x)\cdot n_x}{4\pi\lVert x - y \rVert^3_2} f(y) \dd S_y  = -\frac{\epsilon}{8}\Delta h(0) f(x)+\mathcal{O}(\epsilon^3).
\end{align*}

\item For the 3D Stokeslet, we have 
\begin{align*}
   & \int_{\partial\Omega\cap B_\epsilon(x)}\frac{1}{8\pi}\Bigl(\frac{1}{\lVert x-y\rVert_2}I_3 + \frac{(x-y)(x-y)^T}{\lVert x-y\rVert_2^3}\Bigl)f(y)\dd s(y) 
   \\
   =&
   \frac{\epsilon}{4}f(x) + \frac{1}{8\pi}\int_{B_\epsilon(0)}Q^T\Bigl((s^Ts+h^2(s))^{-3/2}\begin{bmatrix}
       ss^T & h(s)s \\ h(s)s^T & h(s)^2
   \end{bmatrix}\Bigl)Q\bigl(f(x)+\sum_{i=1}^2 \nabla f(x)\cdot\tau_{x,i} s_i\bigr)\dd s +\mathcal{O}(\epsilon^3)
   \\
   =&
   \frac{\epsilon}{4}f(x) + \frac{1}{8\pi}\int_{B_\epsilon(0)}Q^T(s^Ts)^{-3/2}\Bigl(\begin{bmatrix}
       ss^T & 0 \\ 0 & 0
   \end{bmatrix}Q\bigl(f(x)+\sum_{i=1}^2(\nabla f(x)\cdot\tau_{x,i})s_i\bigr)\Bigl)\dd s +\mathcal{O}(\epsilon^3)
   \\
   =&
   \frac{\epsilon}{4}f(x) + \frac{1}{8\pi}\int_{B_\epsilon(0)}Q^T(s^Ts)^{-3/2}\Bigl(\begin{bmatrix}
       ss^T & 0 \\ 0 & 0
   \end{bmatrix}Q f(x) \Bigl)\dd s +\mathcal{O}(\epsilon^3)
   \\
   =&
   \frac{\epsilon}{4}f(x) + \frac{1}{8\pi} Q^T \begin{bmatrix}
       \epsilon \pi I_2 & 0 \\ 0 & 0
   \end{bmatrix}Q f(x)   +\mathcal{O}(\epsilon^3)
   \\
   =&
   \frac{\epsilon}{4}f(x) + \frac{\epsilon}{8} (I_3 - n_x n_x^T) f(x)   +\mathcal{O}(\epsilon^3).
\end{align*}

\end{itemize}

\subsection{Numerical Evaluation of Kernel Integrals}
\label{appendix:panel}
We conclude by describing the numerical evaluation of the 2D kernel integrals
$$ u(x) = \int_{\partial \Omega} \kappa(x - y; n_x, n_y) f(y) \dd S_y, $$
introduced in \cref{ssec:kernel_integral}, using a panel method.
We discretize the curve $\partial \Omega$ into $N$ panels, $\partial \Omega = \cup_{i=1}^{N} \Gamma_i$, ordered counterclockwise. On each panel, both the density $f$ and the solution $u$ are approximated by piecewise constant functions. 
For a given panel $\Gamma$ of length $l$ and orientation angle $\theta$, let $\tau = [\cos\theta,\, \sin\theta]^T$ and $n= [\sin\theta,\,-\cos\theta]^T$ denote the unit tangent and outward unit normal of the panel, respectively. We introduce a local reference coordinate system via the transformation 
\begin{equation}
\label{eq:local_reference_coord_panel}
    \tilde{x} = Q(x - b),\qquad Q^T = \begin{bmatrix}
    \cos\theta & -\sin\theta\\
    \sin\theta & \cos\theta
\end{bmatrix} = \begin{bmatrix}\tau & n\end{bmatrix},
\end{equation}
where $b$ is the starting point of the panel, which is mapped to the origin, and the panel is aligned with the 
$\tilde{x}_1$-axis. 
Writing $\tilde{x} = (\tilde{x}_1, \tilde{x}_2)$, the panel integrals are evaluated in this local reference frame and then transformed back to the global coordinates.
Closed-form expressions for each kernel integral on $\Gamma$ are listed below:

\begin{itemize}
\item For the 2D Laplacian single layer potential, if  $\tilde{x}_2 \neq 0$, then 
\begin{align*}
\int_{\Gamma} \kappa(x-y) \dd S_y &= 
      -\frac{1}{2\pi}\Bigl(
      (l-\tilde{x}_1)\ln\sqrt{(l-\tilde{x}_1)^2 + {\tilde{x}_2}^{2}} + \tilde{x}_1\ln\sqrt{{\tilde{x}_1}^2 + {\tilde{x}_2}^{2}} -l \\ 
      &+\tilde{x}_2\arctan\bigl(\frac{l-\tilde{x}_1}{\tilde{x}_2}\bigr) + \tilde{x}_2\arctan\bigl(\frac{\tilde{x}_1}{\tilde{x}_2}\bigr) 
      \Bigr).
\end{align*}
If $\tilde{x}_2 = 0$, then
\begin{align*}
\int_{\Gamma} \kappa(x - y) \dd S_y 
&= -\frac{1}{2\pi}\Bigl(
      (l-\tilde{x}_1)\ln\sqrt{(l-\tilde{x}_1)^2} + \tilde{x}_1\ln\sqrt{{\tilde{x}_1}^2} - l 
      \Bigr).
\end{align*}

\item For the 2D Laplacian double layer potential, since the panels are ordered counterclockwise, in the local coordinates \cref{eq:local_reference_coord_panel}, the outward unit normal is $\tilde{n}_{\tilde{y}} = (0,-1)$.
If $\tilde{x}_2 \neq 0$, then
\begin{align*}
\label{eq:double_layer_potential}
\int_{\Gamma} \kappa(x - y, n_y) \dd S_y &= 
      -\frac{1}{2\pi}\Bigl[ \arctan(\frac{\tilde{x}_1}{\tilde{x}_2}) +  \arctan(\frac{l-\tilde{x}_1}{\tilde{x}_2})\Bigr]. 
\end{align*}
If $\tilde{x}_2 = 0$, the integral vanishes:
\begin{align*}
\int_{\Gamma} \kappa(x - y, n_y) \dd S_y =  0. 
\end{align*}

\item For the 2D modified Laplacian double layer potential, if $\tilde{x}_2 \neq 0$, then 
\begin{align*}
\int_{\Gamma} \kappa(x - y) \dd S_y &= 
      \frac{1}{2\pi}\Bigl[
      \ln\sqrt{\frac{\tilde{x}_1^2 + {\tilde{x}_2}^{2}}{(l-\tilde{x}_1)^2 + \tilde{x}_2^2}}\,,\, \arctan(\frac{\tilde{x}_1}{\tilde{x}_2}) +  \arctan(\frac{l-\tilde{x}_1}{\tilde{x}_2})\Bigr] Q.
\end{align*}
If $\tilde{x}_2 = 0$, the integral is interpreted in the Cauchy principal value sense and reduces to
\begin{align*}
\int_{\Gamma} \kappa(x - y) \dd S_y &= 
      \frac{1}{2\pi}\Bigl[
      \ln\sqrt{\frac{\tilde{x}_1^2}{(l-\tilde{x}_1)^2}}\,,\, 0\Bigr] Q.
\end{align*}

\item For the 2D adjoint Laplacian double layer potential, if $\tilde{x}_2 \neq 0$, then 
\begin{align*}
\int_{\Gamma} \kappa(x - y, n_x) \dd S_y &= 
      -\frac{1}{2\pi}\Bigl[
      \ln\sqrt{\frac{\tilde{x}_1^2 + {\tilde{x}_2}^{2}}{(l-\tilde{x}_1)^2 + \tilde{x}_2^2}}\,,\, \arctan(\frac{\tilde{x}_1}{\tilde{x}_2}) +  \arctan(\frac{l-\tilde{x}_1}{\tilde{x}_2})\Bigr] Q n_x.
\end{align*}
If $\tilde{x}_2 = 0$, then 
\begin{align*}
\int_{\Gamma} \kappa(x - y, n_x) \dd S_y &= 
      \frac{1}{2\pi}\Bigl[
      \ln\sqrt{\frac{\tilde{x}_1^2}{(l-\tilde{x}_1)^2}}\,,\, 0\Bigr] Q n_x.
\end{align*}

\item For the 2D Stokeslet, if
$\tilde{x}_2 \neq 0$, then
\begin{align*}
\int_{\Gamma} \frac{(x-y)(x-y)^T}{\lVert x - y \rVert_2^2} \dd S_y &= Q^T\begin{bmatrix}
    l - \tilde{x}_2\Bigl(\arctan(\frac{\tilde{x}_1}{\tilde{x}_2}) + \arctan(\frac{l-\tilde{x}_1}{\tilde{x}_2})\Bigr) & \tilde{x}_2
      \ln\sqrt{\frac{\tilde{x}_1^2 + {\tilde{x}_2}^{2}}{(l-\tilde{x}_1)^2 + \tilde{x}_2^2}}\\
      * & \tilde{x}_2\Bigl(\arctan(\frac{\tilde{x}_1}{\tilde{x}_2}) + \arctan(\frac{l-\tilde{x}_1}{\tilde{x}_2})\Bigr)
\end{bmatrix} Q.
\end{align*}
If $\tilde{x}_2 = 0$, then 
\begin{align*}
\int_{\Gamma} \frac{(x-y)(x-y)^T}{\lVert x - y \rVert_2^2} \dd S_y 
&= Q^T\begin{bmatrix}
    l  & 0\\
     0 & 0
      \end{bmatrix} Q.
\end{align*}
\end{itemize}

\section{Floating-Point Cost Analysis}
\label{sec:flops}
In this section, we estimate the number of real floating-point operations required for one
forward evaluation of the M-PCNO. One real addition or multiplication is
counted as one floating point operation. For complex arithmetic, we
count a complex multiplication as six real flops and a complex addition as
two real flops. The evaluation of a scalar pointwise activation
$\sigma$ is assigned a cost of $c_\sigma$ flops.

Let
\[
    X=\{x^{(i)}\}_{i=1}^{N}\subset\calD
\]
denote the point-cloud discretization of $\calD$, and let
$\nu(x^{(i)})$ denote the number of neighbors of $x^{(i)}$. We define
\begin{equation}
    N_f
    :=
    \frac{1}{2}\sum_{i=1}^{N}\nu(x^{(i)})
\end{equation}
as the number of undirected edges in the neighborhood graph. Let $d_g$ be
the latent feature width, and let
\begin{equation}
    K:=(2p+1)^d
\end{equation}
denote the number of retained Fourier modes.

We assume that geometry-dependent quantities, including the quadrature weights, neighbor sets, normal vectors, and normal derivatives, are precomputed and are therefore not included in the inference count.

\paragraph{Lifting and projection}

Assume that the lifting map $\calP$ is a
pointwise affine map. Then its cost is
\begin{equation}
\label{eq:lifting}
    C_{\rm lift} = 2N (d_f+2d) d_g.
\end{equation}
Assume further that the projection map $\calQ$ is a two-layer pointwise multilayer perceptron $\mathbb R^{d_g}\to\mathbb R^{d_g}\to\mathbb R^{d_u}$ with one pointwise activation in the hidden layer. Its cost is
\begin{equation}
\label{eq:projection}
C_{\rm proj} = 2N d_g^2 + 2N d_g d_u + c_{\sigma} N d_g.
\end{equation}

\paragraph{Long-range operator}
To reduce memory usage, the Fourier basis functions $e^{2\pi i k \cdot \frac{x}{2l}}$ are evaluated on the fly rather than stored for
all modes and points.   
For each retained mode \(k\), evaluating the basis on all points costs $4 d N$ flops. Forming the weighted modal coefficient costs \(4d_g N\) flops, applying the learned complex matrix \(W_v^k\) costs \(8d_g^2-2d_g\) flops, and reconstructing the output at all points costs \(8d_g N\) flops. 
The source- and target-normal factors in
\cref{eq:K_long_surf} introduce additional pointwise operations. Forming
$g(y)\otimes n_y$ and
$(\bK_{\rm long}^{(1)}g)(x)\otimes n_x$ costs
$2dd_gN$ flops. Each of the matrices $W_1$ and $W_2$ maps
$d_g(d+1)$ input features to $d_g$ output features and therefore costs
$\bigl(2d_g^2(d+1)-d_g\bigr)N$ flops. 
Consequently, the total cost of the long-range operator is
\begin{equation}
\label{eq:cost-unstructured-integral}
K (12d_g + 4d) N + K(8d_g^2 - 2d_g) + 2\bigl( (2d+2) d_g + d - 1\bigr)d_g N.
\end{equation}

\paragraph{Short-range operator}
The dominant operation in the short-range operator is the local
least-squares gradient reconstruction. For each directed edge, forming the
feature difference costs $d_g$ flops, while applying the local aggregation
weights costs $2dd_g$ flops. If the undirected neighborhood graph contains $N_f$
edges, and hence $2N_f$ directed edges, the gradient reconstruction costs
$2(2d+1)d_gN_f$.
The auxiliary short-range term $\bK_{\rm short}^{(1)}$ contains a pointwise affine map, a gradient transformation, one \texttt{SoftSign} transformation, a corresponding channel-mixing matrix, and a vector addition. Counting three flops per scalar \texttt{SoftSign} evaluation, these operations cost
$\bigl((2d+4)d_g^2+2d_g\bigr)N$.
The geometry-dependent branch contains the transformation of the normal and normal gradient features,  a second \texttt{SoftSign} transformation, a transformation of $g(x)$,
a componentwise product, the application of $W_{g,4}$, and the addition to
$\bK_{\rm short}^{(1)}g$. Its cost is $\bigl(4d_g^2 + 2(d^2+d)d_g + 2d_g\bigr)N$.
Consequently, the total cost of the short-range operator is
\begin{equation}
2(2d+1)d_gN_f + 2\bigl((d+4)d_g + d^2+d + 2\bigr)d_g N.
\end{equation}

\paragraph{Multiscale point cloud neural layer}
In addition to evaluating the long- and short-range operators, each multiscale layer
adds their outputs, applies the pointwise activation, and adds the
residual connection. These operations
cost $(2+c_\sigma)d_gN$.
Therefore, the cost of one multiscale point-cloud neural layer is
\begin{equation}
\label{eq:cost-mpcno-layer}
\begin{split}
    C_{\rm layer} =    
    & K (12d_g + 4d) N + 2(2d+1)d_gN_f + K(8d_g^2 - 2d_g) \\
    &+ \bigl( (6d+12)d_g + 2d^2+4d + 2 + c_\sigma\bigr)d_g N.
\end{split}
\end{equation}

\paragraph{Total inference cost.}
We assume that the neighborhood size is uniformly bounded and is comparable to that of a structured grid, so that  \(N_f \approx d N\). 
We further assume that $c_\sigma=\mathcal{O}(1)$ and 
$d_f,d_u, d \ll d_g$. 
Combining \cref{eq:lifting,eq:projection,eq:cost-mpcno-layer} and retaining the dominant terms in $N$, $K$, $d_g$, and $L$, we obtain
\begin{equation}
\begin{split}
    C_{\rm total} &= C_{\rm lift} + C_{\rm proj} + L C_{\rm layer} \\
                 &=  \calO\bigl(12 K L d_g N + 8 L K d_g^2 + (2 + (6d+12)L) d_g^2 N\bigr).
\end{split}
\end{equation}
Since $K=(2p+1)^d$, the inference cost is linear in $N$ for fixed spatial
dimension $d$, Fourier truncation parameter $p$, latent width $d_g$, and
network depth $L$.

\section{Experimental Detail}
\label{app:experimental_settings}
In this section, we give details about all numerical experiments.

\subsection{Common Training Setting}

All experiments involve variable computational domains discretized with different numbers of points. To enable consistent and parallelizable processing, all samples (inputs and outputs) are zero-padded to a fixed maximum length. Unless otherwise stated, the M-PCNO architecture consists of 5 multiscale point cloud neural layers~\eqref{eq:mpcno-layer}, each with $d_g = 64$ channels. Training is performed using the Adam optimizer~\cite{kingma2015adam} with $(\beta_1,\beta_2)=(0.9,0.999)$, a base learning rate of $5\times 10^{-4}$, and weight decay $10^{-4}$. The learning rate is scheduled by OneCycleLR~\cite{smith2019super} with 
\texttt{div\_factor} $=2$, \texttt{final\_div\_factor} $=100$, and \texttt{pct\_start} $=0.2$.
The training is with a batch size of $8$ over 500 epochs. All experiments are run on NVIDIA A100 80G GPUs.


\subsection{Kernel Integral Problem}
\label{app:kernel_integral_details}
For the two-dimensional curve experiments, each geometry is represented by a closed polygonal curve. The curves are generated in polar form. Specifically, for $\theta\in[0,2\pi)$, we sample a random radius
\begin{equation}
    r(\theta)
    =
    \tanh\left(
        r_0
        +
        \sum_{j=1}^{k}
        \left(
            a_j\sin(j\theta)+b_j\cos(j\theta)
        \right)
    \right)
    +1.5,
\end{equation}
where $r_0$, $a_j$, and $b_j$ are randomly sampled coefficients, with higher-frequency coefficients scaled by $1/\sqrt{j}$. The corresponding curve is given by
\begin{equation}
    \bm{x}(\theta)=r(\theta)(\cos\theta,\sin\theta).
\end{equation}
The generated curves are further rescaled so that they fit into a fixed computational box. For the geometric generalization test, we also generate two-component geometries consisting of two independently sampled closed curves.

All curves lie within a bounding box with side lengths $l_1=l_2=5$, which are also used as the characteristic length scales of the Fourier basis.
Input functions are sampled from a Gaussian random field, and the reference
integrals are evaluated by the panel method in \cref{appendix:panel}. For the
integral-operator and exterior Laplace experiments, we use $8000$ training
samples and $1000$ single-curve test samples. An additional test set with
$1000$ two-component geometries is used to evaluate geometric and
topological generalization.

The complete relative $L^2$ test errors for the linear model and the 5-layer M-PCNO are reported in \cref{tab:single-layer,tab:M-PCNO-pn-vs-error},
respectively. These values are visualized in
\cref{fig:kernel_integral_error_plot} of the main paper. Within each table
entry, the first value is the single-curve error and the second is the
two-curve error.

\begin{table}[htbp]
    \begin{center}
    \resizebox{0.95\textwidth}{!}{%
        \begin{tabular}{c|cccc}
			\Xhline{1.2pt}
			\diagbox{Kernel}{$p$}  & $8 $ & $16$ & $32$ & $64$\\ 
            \Xhline{1.2pt}
			\makecell{Laplacian single \\layer potential}   & 0.4251\quad 0.7979   & 0.1117\quad 0.2211   & 0.0597\quad 0.081   & 0.0593\quad 0.0376   \\ \hline
			\makecell{Laplacian double \\layer potential}   & 10.1761\quad 22.6872   & 4.5952\quad 13.3501   & 1.6056\quad 6.3796   & 0.5962\quad 3.2029   \\ \hline
			\makecell{Modified Laplacian \\double layer potential}   & 10.6008\quad 18.897   & 4.8453\quad 11.8611   & 1.7733\quad 5.9475   & 0.6173\quad 2.5156   \\ \hline
			\makecell{Adjoint Laplacian \\double layer potential}   & 5.4653\quad 10.792   & 2.2898\quad 6.0529   & 0.7547\quad 2.8915   & 0.2933\quad 1.2417   \\ \hline
			\makecell{Stokeslet}   & 1.3103\quad 2.3547   & 0.3562\quad 0.7624   & 0.1418\quad 0.2538   & 0.1085\quad 0.1161   \\ 
            \Xhline{1.2pt}
		\end{tabular}
        }
    \end{center}
    \caption{Kernel integral learning with the single-layer linear model. Relative $L^2$ test errors ($\times 10^{-2}$) for different truncated mode numbers $p$, trained on $n=8000$ single-curve samples. Each entry reports single-curve and two-curve test errors.}
    \label{tab:single-layer}
\end{table}

\begin{table}[htbp]
    \begin{center}
    \resizebox{0.95\textwidth}{!}{%
        \begin{tabular}{c|cccc}
			\Xhline{1.2pt}
			\diagbox{Kernel}{$p$}  & $8 $ & $16$ & $32$ & $64$\\ \Xhline{1.2pt}
			\makecell{Laplacian single \\layer potential}   & 0.291\quad 0.7256   & 0.0934\quad 0.2096   & 0.0757\quad 0.0828   & 0.1434\quad 0.1866   \\ \hline
			\makecell{Laplacian double \\layer potential}   & 1.3531\quad 6.1345   & 0.6054\quad 3.281   & 0.2735\quad 1.8418   & 0.2287\quad 1.6217   \\ \hline
			\makecell{Modified Laplacian \\double layer potential}   & 0.9628\quad 4.0536   & 0.5433\quad 2.8021   & 0.2857\quad 1.4363   & 0.1993\quad 0.8351   \\ \hline
			\makecell{Adjoint Laplacian \\double layer potential}   & 1.0068\quad 4.3489   & 0.5354\quad 2.7218   & 0.38\quad 1.8117   & 0.1632\quad 0.7145   \\ \hline
			\makecell{Stokeslet}   & 0.5808\quad 1.6941   & 0.2254\quad 0.5793   & 0.1388\quad 0.2525   & 0.1477\quad 0.2376   \\ \Xhline{1.2pt}
		\end{tabular}
        }
    \end{center}
    \begin{center}
    \resizebox{0.95\textwidth}{!}{%
        \begin{tabular}{c|cccc}
			\Xhline{1.2pt}
			\diagbox{Kernel}{$n$}  & $1000$ & $2000$ & $4000$ & $8000$\\ \Xhline{1.2pt}
			\makecell{Laplacian single \\layer potential}   & 0.89\quad 1.1019   & 0.3475\quad 0.4173   & 0.1719\quad 0.268   & 0.0757\quad 0.0828   \\ \hline
			\makecell{Laplacian double \\layer potential}   & 0.9926\quad 3.8675   & 0.5247\quad 2.7039   & 0.3998\quad 2.4095   & 0.2735\quad 1.8418   \\ \hline
			\makecell{Modified Laplacian \\double layer potential}   & 0.5943\quad 2.2242   & 0.4147\quad 1.7736   & 0.3445\quad 1.6432   & 0.2857\quad 1.4363   \\ \hline
			\makecell{Adjoint Laplacian \\double layer potential}   & 0.575\quad 2.1708   & 0.4826\quad 1.9921   & 0.4178\quad 1.8831   & 0.38\quad 1.8117   \\ \hline
			\makecell{Stokeslet}   & 0.8023\quad 1.1355   & 0.3857\quad 0.5847   & 0.2187\quad 0.377   & 0.1388\quad 0.2525   \\ \Xhline{1.2pt}
		\end{tabular}
        }
    \end{center}
    \caption{Kernel integral learning with the 5-layer M-PCNO. Relative $L^2$ test errors ($\times 10^{-2}$) for different truncated mode numbers $p$, trained on $n=8000$ single-curve samples (top), and for different single-curve training dataset sizes $n$ with $p=32$ fixed (bottom). Each entry reports single-curve and two-curve test errors.}\label{tab:M-PCNO-pn-vs-error}
\end{table}

\subsection{Exterior Neumann Problem}
\label{app:exterior_neumann_details}

The curve distributions, input sampling, training/test sizes, and common training settings are the same as in
\cref{app:kernel_integral_details}.
The complete relative $L^2$ test errors for the 5-layer M-PCNO are reported in \cref{tab:exterior_neumann},
respectively. These values are visualized in
\cref{fig:exterior_neumann} of the main paper.

\begin{table}[htbp]
\begin{center}
\resizebox{0.85\textwidth}{!}{%
		\begin{tabular}{c|cccc}
			\Xhline{1.1pt}
			\diagbox{$p$}{$n$}  & $1000$ & $2000$ & $4000$ & $8000$\\ \Xhline{1.1pt}
            8 & 3.4427\quad 13.5705   & 2.2997\quad 13.5514   & 1.6629\quad 13.0614   & 1.2328\quad 13.024   \\ \hline
			16 & 2.7867\quad 10.0941   & 1.7473\quad 8.8995   & 1.2366\quad 8.6305   & 0.9289\quad 8.1034   \\ \hline
			32 & 3.5967\quad 9.3321   & 2.1442\quad 7.3943   & 1.2108\quad 5.764   & 0.9217\quad 4.9039   \\ \hline
			64 & 4.5981\quad 10.7413   & 2.9173\quad 9.52   & 1.7887\quad 7.917   & 1.0939\quad 6.1703   \\ \Xhline{1.1pt}
		\end{tabular}
        }
	\end{center}
    \caption{Neumann-to-Dirichlet map for the exterior Laplacian learned with a 5-layer M-PCNO. Relative $L^2$ test errors ($\times 10^{-2}$) for different single-curve training dataset sizes $n$ and truncated mode numbers $p$. Each entry reports single-curve and two-curve test errors.}
    \label{tab:exterior_neumann}
\end{table}

\subsection{Potential Flow Problem}
For the 3D potential-flow experiment, we generate a mixed geometry dataset consisting of car and aircraft surfaces. 
The dataset consists of two categories: cars and aircraft. The car subset (about 5,000 samples), including fastback, notchback, and estateback designs, is derived from DrivAerNet++\cite{elrefaie2025drivaernetlargescalemultimodalcar} (see \cref{fig:potential_flow}). The aircraft subset (about 6,000 samples) is generated using NASA's Open Vehicle Sketch Pad (OpenVSP). 
Baseline aircraft designs, including a fighter jet, a turboprop aircraft, and a commercial airliner,
are obtained from the OpenVSP Airshow (https://airshow.openvsp.org/) and modified to form reference templates. Stochastic perturbations are then applied to the shape parameters of these templates, yielding a diverse range of aerodynamic configurations. All surface meshes are converted to triangular meshes and decimated to about $40,000$ elements, then isotropically scaled to fit within the bounding box $[-1, 1]^3$. The characteristic length scales of the Fourier basis are $(l_1, l_2, l_3) = (2.05, 2.05, 0.75)$.

\subsection{Turbulent Flow Problem}

The \textit{ShapeNet-Car} benchmark \cite{umetani2018learning} contains surface-pressure fields from steady external-flow simulations around $611$ vehicle geometries. Each vehicle is represented by a triangular surface mesh with approximately $3.7\times10^3$ points. Following the benchmark protocol used by GINO, we use $500$ samples for training and $111$ for testing. 
For the reported M-PCNO result, we use $p=16$, characteristic length scales $(l_1, l_2, l_3) = (2, 2, 6)$ for the Fourier basis, and 7 layers with hidden width $d_g = 64$. This configuration gives a test relative $L^2$ error of $0.0636$.

To assess sensitivity to random initialization, we repeated the training with five different random seeds. The resulting training and test errors are shown in \cref{fig:car-error-epoch}. The small variation across runs indicates that the observed performance is not highly sensitive to random initialization. 

\begin{figure}[htbp]
     \centering
     \includegraphics[width=0.6\textwidth]{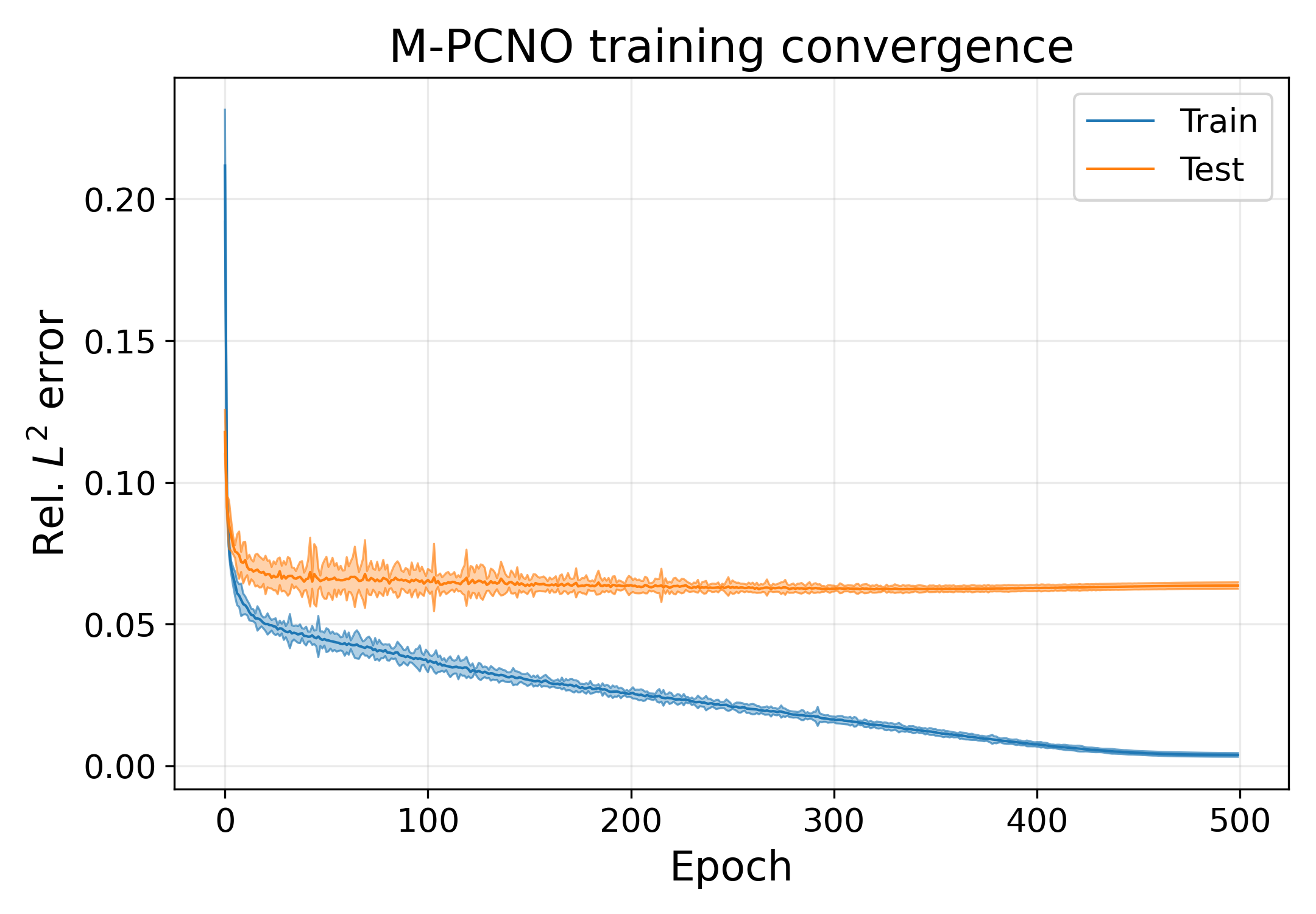}\\
     \caption{Training and test errors versus the number of epochs for the turbulent flow problem. The curves show the mean over
    five runs with different random seeds, and the error bars indicate three standard deviation.}
     \label{fig:car-error-epoch}
\end{figure}

\bibliographystyle{siamplain}
\bibliography{ref}
\end{document}